\documentclass[a4paper,10pt,reqno]{article}

\usepackage[utf8]{inputenc}
\usepackage[T1]{fontenc}
\usepackage{lmodern}
\usepackage[english]{babel}
\usepackage{microtype}

\usepackage{amsmath,amssymb,amsfonts,amsthm}
\usepackage{mathtools,accents}
\usepackage{mathrsfs}
\usepackage{aliascnt}
\usepackage{braket}
\usepackage{bm}

\usepackage[a4paper,margin=3cm]{geometry}
\usepackage[citecolor=blue,colorlinks]{hyperref}

\usepackage{enumerate}
\usepackage{xcolor}

\makeatletter
\g@addto@macro\@floatboxreset\centering
\makeatother

\makeatletter
\def\newaliasedtheorem#1[#2]#3{
	\newaliascnt{#1@alt}{#2}
	\newtheorem{#1}[#1@alt]{#3}
	\expandafter\newcommand\csname #1@altname\endcsname{#3}
}
\makeatother

\numberwithin{equation}{section}

\newtheoremstyle{slanted}{\topsep}{\topsep}{\slshape}{}{\bfseries}{.}{.5em}{}

\theoremstyle{plain}
\newtheorem{theorem}{Theorem}[section]
\newaliasedtheorem{proposition}[theorem]{Proposition}
\newaliasedtheorem{lemma}[theorem]{Lemma}
\newaliasedtheorem{corollary}[theorem]{Corollary}
\newaliasedtheorem{counterexample}[theorem]{Counterexample}

\theoremstyle{definition}
\newaliasedtheorem{definition}[theorem]{Definition}
\newaliasedtheorem{question}[theorem]{Question}
\newaliasedtheorem{conjecture}[theorem]{Conjecture}

\theoremstyle{remark}
\newaliasedtheorem{remark}[theorem]{Remark}
\newaliasedtheorem{example}[theorem]{Example}

\newcommand{\setN}{\mathbb{N}}

\newcommand{\setR}{\mathbb{R}}

\newcommand{\eps}{\varepsilon}

\let\altphi\phi
\let\phi\varphi
\let\varphi\altphi
\let\altphi\undefined

\newcommand{\abs}[1]{\left\lvert#1\right\rvert}
\newcommand{\norm}[1]{\left\lVert#1\right\rVert}

\DeclareMathOperator{\Hess}{Hess}

\newcommand{\di}{\mathop{}\!\mathrm{d}}

\newcommand{\res}{\mathop{\hbox{\vrule height 7pt width .5pt depth 0pt
			\vrule height .5pt width 6pt depth 0pt}}\nolimits}

\DeclareMathOperator{\supp}{supp}

\newcommand{\Ch}{{\sf Ch}}

\DeclareMathOperator{\Lip}{Lip}
\DeclareMathOperator{\Lipb}{Lip_b}

\DeclareMathOperator{\lip}{lip} 
\DeclareMathOperator{\Ric}{Ric}

\DeclareMathOperator{\diam}{diam}

\DeclareMathOperator{\de}{d}

\newcommand{\leb}{\mathscr{L}}

\newcommand{\Prob}{\mathscr{P}}

\newcommand{\XX}{{\boldsymbol{X}}}

\newcommand{\dist}{\mathsf{d}}

\newcommand{\meas}{\mathfrak{m}}

\newcommand{\Tan}{{\rm Tan}}

\DeclareMathOperator{\CD}{CD}
\DeclareMathOperator{\RCD}{RCD}
\DeclareMathOperator{\ncRCD}{ncRCD}
\newfont{\tmpf}{cmsy10 scaled 2500}

\def\XXint#1#2#3{{\setbox0=\hbox{$#1{#2#3}{\int}$ }
		\vcenter{\hbox{$#2#3$ }}\kern-.6\wd0}}

\begin{document}
	
	\title{Volume bounds for the quantitative singular strata of non collapsed $\RCD$ metric measure spaces}
	\author{Gioacchino Antonelli
		\thanks{Scuola Normale Superiore, \url{gioacchino.antonelli@sns.it}.}\and
		Elia Bru\'e 
		\thanks{Scuola Normale Superiore, \url{elia.brue@sns.it}.}\and Daniele Semola
		\thanks{Scuola Normale Superiore,
			\url{daniele.semola@sns.it}.}}
	\maketitle
		
	\begin{abstract}
	The aim of this note is to generalize to the class of non collapsed $\RCD(K,N)$ metric measure spaces the volume bound for the effective singular strata obtained by Cheeger and Naber for non collapsed Ricci limits in \cite{CheegerNaber13a}. The proof, which is based on a quantitative differentiation argument, closely follows the original one.  
	As a simple outcome we provide a volume estimate for the enlargement of Gigli-DePhilippis' boundary (\cite[Remark 3.8]{DePhilippisGigli18}) of $\ncRCD(K,N)$ spaces.
	\end{abstract}

	\tableofcontents

\section*{Introduction}

In the last years the theory of metric measure spaces $(X,\dist,\meas)$ satisfying the Riemannian curvature dimension condition has undergone several remarkable developments. After the introduction, in the independent works \cite{Sturm06a,Sturm06b} and \cite{LottVillani}, of the curvature dimension condition $\CD(K,N)$ encoding in a synthetic way the notion of Ricci curvature bounded from below and dimension bounded above, the definition of $\RCD(K,N)$ metric measure space was proposed and extensively studied in \cite{Gigli13,ErbarKuwadaSturm15,AmbrosioMondinoSavare15} (see also \cite{CavallettiMilman} for the equivalence between the $\RCD^*(K,N)$ and the $\RCD(K,N)$ condition) in order to single out spaces with Hilbert-like behaviour at infinitesimal scale. The infinite dimensional counterpart of this notion had been previously investigated in \cite{AmbrosioGigliSavare14}.\\ 
In particular, due to the compatibility of the $\RCD$ condition with the smooth case of Riemannian manifolds with Ricci curvature bounded form below and to its stability with respect to pointed measured Gromov-Hausdorff convergence, limits of smooth Riemannian manifolds with Ricci curvature uniformly bounded from below and dimension uniformly bounded from above are $\RCD(K,N)$ spaces. The study of Ricci limits was initiated by Cheeger and Colding in the nineties in the series of papers \cite{CheegerColding96,CheegerColding97,CheegerColding2000a,CheegerColding2000b} and has seen remarkable developments in more recent years  (see for instance \cite{ColdingNaber}).
Since the above mentioned pioneering works, it was known that the regularity theory for Ricci limits improves adding to the lower curvature bound a uniform lower bound for the volume of unit balls along the converging sequence of Riemannian manifolds: this is the case of the so called \textit{non collapsed} Ricci limits. In particular, as a consequence of the volume convergence theorem proved in \cite{Colding97}, it is known that the limit measure of the volume measures is the Hausdorff measure on the limit metric space (while this might not be the case for a general Ricci limit space).

Inspired by the theory of non collapsed Ricci limits, De Philippis and Gigli proposed in \cite{DePhilippisGigli18} a notion of non collapsed $\RCD(K,N)$ metric measure space $(X,\dist,\meas)$ ($\ncRCD(K,N)$ for short) asking that $\meas=\mathcal{H}^N$, the $N$-dimensional Hausdorff measure over $(X,\dist)$. Let us remark that this class of spaces had already been studied by Kitabeppu in \cite{Kitabeppu17}.\\ 
Let us point out that recently examples of metric measure spaces which are $\ncRCD$ but not non collapsed Ricci limits have been built: hence a gap widens between the two theories. Nevertheless in \cite{DePhilippisGigli18} the authors were able to prove that many of the structural results valid for non collapsed Ricci limits hold for $\ncRCD$ spaces. In particular, building upon \cite{DePhilippisGigli16}, it is possible to prove that any tangent cone to a $\ncRCD$ space is a metric cone. Letting then $\mathcal{R}\subset X$ be the set of those points where the tangent cone is the $N$-dimensional Euclidean space, following \cite{CheegerColding97} it is possible to introduce a stratification
\begin{equation*}
\mathcal{S}^0\subset\dots\subset\mathcal{S}^{N-1}=\mathcal{S}=X\setminus\mathcal{R}
\end{equation*}
of the singular set $\mathcal{S}$, where, for any $k=0,\dots,N-1$, $\mathcal{S}^k$ is the set of those points where no tangent cone splits a factor $\setR^{k+1}$. Adapting the arguments of \cite{CheegerColding97}, in \cite{DePhilippisGigli18} the Hausdorff dimension estimate $\dim_{H}\mathcal{S}^k\le k$ was obtained.

In \cite{CheegerNaber13a} a quantitative and effective counterpart of the above mentioned stratification of the singular set was introduced letting, for any $k=0,\dots,N-1$ and for any $r,\eta>0$, $\mathcal{S}^{k}_{\eta,r}$ be the set of those points $x\in X$ where the scale invariant Gromov-Hausdorff distance between the ball $B_s(x)$ and any ball of the same radius centered at the tip of a metric cone splitting a factor $\setR^{k+1}$ is bigger than $\eta $ for any $r<s<1$.\\     
While in the classical stratification points are separated according to the number of symmetries of tangent cones, in the quantitative one they are classified according to the number of symmetries of balls of fixed scales therein centered. In particular, the effective singular strata might be non empty even in the case of smooth Riemannian manifolds while in that case there is no singular point.

Starting from \cite{CheegerNaber13a} a number of properties for the effective singular strata on non collapsed Ricci limit spaces have been obtained. In particular, in the very recent \cite{CheegerJangNaber18}, the authors were able to prove $k$-rectifiability of the classical singular stratum $\mathcal{S}^k$ building on the top of some new volume estimates for the effective strata.  

The aim of this note is twofold. On the one hand our main result \autoref{thm:RCDquantitativevolumebound} generalizes to the class of $\ncRCD$ the volume estimate for the effective singular strata obtained by Cheeger and Naber in \cite{CheegerNaber13a} (which is easily seen to be stronger than the above mentioned Hausdorff dimension estimate $\dim_{H}\mathcal{S}^k\le k$), on the other hand we give detailed proofs (in the metric context) of some of the results that therein were just stated. Let us point out that \autoref{thm:RCDquantitativevolumebound} has already an application in the proof of \cite[Theorem 5.8]{MondinoKapovitch19}. 

Let us remark that the proof of the volume estimate, which closely follows the one for Ricci limits, provides an instance of the so called \textit{quantitative differentiation} technique that, although being quite recent in its formulation, has already a broad range of applications in the regularity theory in various different geometric and analytic contexts.\\ 
In general, quantitative differentiation allows to bound the number of locations and scales at which a given geometric configuration is far away from any element of a class of special configurations. In the case of our interest special configurations are the conical ones.
We refer to \cite{Cheeger12} for a general survey about quantitative differentiation and detailed list of references to the recent applications of this tools in the various contexts.

This note is organised as follows: in \autoref{sec:preliminaries} we list a few basic definitions and results useful when dealing with $\ncRCD$ metric measure spaces. Most of the results are stated without proof and references are indicated. We provide instead proofs for the ``almost volume cone implies almost metric cone'' \autoref{FromAlmostVolumeCone} and the ``almost cone splitting'' \autoref{thm:conesplittingquant}, since we were not able to find any reference in the literature. In \autoref{sec:mainresult} we give a complete proof of the volume bound for the effective singular strata following the same strategy introduced by Cheeger and Naber in the setting of non collapsed Ricci limit spaces.


\textbf{Acknowledgements.} The authors warmly thank Luigi Ambrosio for several discussions around the topic of this note. They are also grateful to Andrea Mondino for some useful comments on an earlier version of the paper.

\section{Preliminaries}\label{sec:preliminaries}
Throughout this paper a \textit{metric measure space} is a triple $(X,\dist,\meas)$, where $(X,\dist)$ is a separable metric space and $\meas$ is a nonnegative Borel measure on $X$ finite on bounded sets. From now on when we will write \emph{m.m.s.} we mean \emph{metric measure space(s)}. We will denote by $B_r(x)=\{\dist(\cdot,x)<r\}$ and $\bar{B}_r(x)=\{\dist(\cdot,x)\leq r\}$ the open and closed balls respectively, by $\Lip(Z)$ (resp. $\Lipb(Z)$) the space of Lipschitz (resp. bounded) functions and for any $f\in\Lip(Z)$ we shall denote its slope by
\begin{equation*}
\lip f(x)\doteq \limsup_{y\to x}\frac{\abs{f(x)-f(y)}}{\dist(x,y)}.
\end{equation*}
We will use the standard notation $L^p(X,\meas)$, for the $L^p$ spaces and $\leb^n,\mathcal{H}^n$ for the $n$-dimensional Lebesgue measure on $\setR^n$ and the $n$-dimensional Hausdorff measure on a metric space, respectively. We shall denote by $\omega_n$ the Lebesgue measure of the unit ball in $\setR^n$.

The Cheeger energy $\Ch:L^2(X,\meas)\to[0,+\infty]$ associated to a m.m.s. $(X,\dist,\meas)$ is the convex and lower semicontinuous functional defined through
\begin{equation}\label{eq:cheeger}
\Ch(f)\doteq \inf\left\lbrace\liminf_{n\to\infty}\int\lip^2f_n\di\meas:\quad f_n\in\Lipb(X)\cap L^2(X,\meas),\ \norm{f_n-f}_2\to 0 \right\rbrace
\end{equation}
and its finiteness domain will be denoted by $W^{1,2}(X,\dist,\meas)$. Looking at the optimal approximating sequence in \eqref{eq:cheeger}, it is possible to identify a canonical object $\abs{\nabla f}$, called minimal relaxed slope, providing the integral representation
\begin{equation*}
\Ch(f)\doteq \int_X\abs{\nabla f}^2\di\meas\qquad\forall f\in W^{1,2}(X,\dist,\meas).
\end{equation*}
Any metric measure space such that $\Ch$ is a quadratic form is said to be \textit{infinitesimally Hilbertian} and from now on we shall always make this assumption, unless otherwise stated. Let us recall from \cite{AmbrosioGigliSavare14,Gigli15} that, under this assumption, the function 
\begin{equation*}
\nabla f_1\cdot\nabla f_2\doteq \lim_{\eps\to 0}\frac{\abs{\nabla(f_1+\eps f_2)}^2-\abs{\nabla f_1}^2}{2\eps}
\end{equation*}
defines a symmetric bilinear form on $W^{1,2}(X,\dist,\meas)\times W^{1,2}(X,\dist,\meas)$ with values into $L^1(X,\meas)$.

It is possible to define a Laplacian operator $\Delta:\mathcal{D}(\Delta)\subset L^{2}(X,\meas)\to L^2(X,\meas)$ in the following way. We let $\mathcal{D}(\Delta)$ be the set of those $f\in W^{1,2}(X,\dist,\meas)$ such that, for some $h\in L^2(X,\meas)$, one has
\begin{equation}\label{eq:amb1}
\int_X \nabla f\cdot\nabla g\di\meas=-\int_X hg\di\meas\qquad\forall g\in W^{1,2}(X,\dist,\meas) 
\end{equation} 
and, in that case, we put $\Delta f=h$. 
It is easy to check that the definition is well-posed and that the Laplacian is linear (because $\Ch$ is a quadratic form).\\

\subsection{$\RCD(K,N)$ metric measure spaces}\label{section: preliminaries RCD}

The notion of $\RCD(K,N)$ m.m.s. was proposed and extensively studied in \cite{Gigli15,AmbrosioMondinoSavare15,ErbarKuwadaSturm15} (see also \cite{CavallettiMilman} for the equivalence between the $\RCD$ and the $\RCD^*$ condition), as a finite dimensional counterpart to $\RCD(K,\infty)$ m.m.s. which were introduced and firstly studied in \cite{AmbrosioGigliSavare14} (see also \cite{AmbGigliMondRaj12}, dealing with the case of $\sigma$-finite reference measures). We point out that these spaces can be introduced and studied both from an Eulerian point of view, based on the so-called $\Gamma$-calculus, and from a Lagrangian point of view, based on optimal transportation techniques, which is the one we shall adopt in this brief introduction. 

Let us start recalling the so-called curvature dimension condition $\CD(K,N)$. Its introduction dates back to the seminal and independent works \cite{LottVillani} and \cite{Sturm06a,Sturm06b}, while in this presentation we closely follow \cite{BacherSturm10}.

\begin{definition}[Curvature dimension bounds]\label{def:CD}
	Let $K\in\setR$ and $1\le N<+\infty$. We say that a m.m.s. $(X,\dist,\meas)$ is a $\CD(K,N)$ space if, for any $\mu_0,\mu_1\in\Prob(X)$ absolutely continuous w.r.t. $\meas$ with bounded support, there exists an optimal geodesic plan $\Pi\in\Prob(\text{Geo}(X))$ such that for any $t\in[0,1]$ and for any $N'\ge N$ we have 
	\begin{align*}
	-\int & \rho_t^{1-\frac{1}{N'}}\di\meas\\
	\le &-\int\left\lbrace\tau_{K,N'}^{(1-t)}(\dist(\gamma(0),\gamma(1)))\rho_0^{-\frac{1}{N'}}(\gamma(0))+\tau_{K,N'}^{(t)}(\dist(\gamma(0),\gamma(1)))\rho_1^{-\frac{1}{N'}}(\gamma(1)) \right\rbrace\di\Pi(\gamma), 
	\end{align*} 
	where $(e_t)_{\sharp}\Pi=\rho_t\meas$, $\mu_0=\rho_0\meas$, $\mu_1=\rho_1\meas$ and the distortion coefficients $\tau_{K,N}^{t}(\cdot)$ are defined as follows. First we define the coefficients $[0,1]\times[0,+\infty)\ni(t,\theta)\mapsto\sigma_{K,N}^{(t)}(\theta)$ by
	\begin{equation*}
	\sigma_{K,N}^{(t)}(\theta)\doteq
	\begin{cases*}
	+\infty &\text{if $K\theta^2\ge N\pi^2$,}\\
	\frac{\sin(t\theta\sqrt{K/N})}{\sin(\theta\sqrt{K/N})}&\text{if $0<\theta<N\pi^2$,}\\
	t&\text{if $K\theta^2=0$,}\\
	\frac{\sinh(t\theta\sqrt{K/N})}{\sinh(\theta\sqrt{K/N})}&\text{if $K\theta^2<0$,}
	\end{cases*}
	\end{equation*}
	then we set $\tau_{K,N}^{(t)}(\theta)\doteq t^{1/N}\sigma_{K,N-1}^{(t)}(\theta)^{1-1/N}$.
\end{definition}

The main object of our study in this paper will be $\RCD(K,N)$ spaces, that we introduce below.

\begin{definition}\label{defRCDKN}
	We say that a metric measure space $(X,\dist,\meas)$ satisfies the \textit{Riemannian curvature-dimension} condition (it is an $\RCD(K,N)$ m.m.s. for short) for some $K\in\setR$ and $1\le N<+\infty$ if it is a $\CD(K,N)$ m.m.s. and the Banach space $W^{1,2}(X,\dist,\meas)$ is Hilbert.
\end{definition}
Note that, if $(X,\dist,\meas)$ is an $\RCD(K,N)$ m.m.s., then so is $(\supp\meas,\dist,\meas)$, hence in the following we will always tacitly assume $\supp\meas = X$.

We assume the reader to be familiar with the notion of pointed measured Gromov Hausdorff convergence (pmGH-convergence for short), referring to \cite[Chapter 27]{Villani09} for an overview on the subject. 

\begin{remark}\label{remark:stability}
	A fundamental property of $\RCD(K,N)$ spaces, that will be used several times in this paper, is the stability w.r.t. pmGH convergence, meaning that a pmGH limit of a sequence of (pointed) $\RCD(K,N)$ spaces is still an $\RCD(K,N)$ m.m.s.. 
\end{remark}

We recall that any $\RCD(K,N)$ m.m.s. $(X,\dist,\meas)$ satisfies the Bishop-Gromov inequality:
\begin{equation}\label{prop:BishopGromovInequality}
\frac{\meas(B_R(x))}{v_{K,N}(R)}\le\frac{\meas(B_r(x))}{v_{K,N}(r)},
\end{equation}
for any $0<r<R$ and for any $x\in X$, where $v_{K,N}(r)\doteq N\omega_N\int_0^r \left(s_{K,N}(s)\right)^{N-1}\de s$ and
\begin{equation}\label{DefinitionSkn}
	s_{K,N}(r)\doteq
	\begin{cases}
	\sqrt{\frac{N-1}{K}}\sin\left(\sqrt{\frac{K}{N-1}}r\right) & \mbox{if} \quad K>0, \\
	r & \mbox{if} \quad K=0, \\
	\sqrt{\frac{N-1}{-K}}\sinh\left(\sqrt{\frac{-K}{N-1}}r\right) & \mbox{if} \quad K<0. \\
	\end{cases}
\end{equation}
In particular $(X,\dist,\meas)$ is locally uniformly doubling, that is to say, for any $R>0$ there exists $C_R>0$ such that
\begin{equation}\label{eq:locdoubling}
\meas(B_{2r}(x))\le C_R\meas(B_r(x))\quad \text{for any $x\in X$ and for any $0<r<R$.}
\end{equation}

%

We refer to \cite[Theorem 30.11]{Villani09} for the proof of the Bishop-Gromov inequality in the setting of metric measure spaces satisfying the curvature dimension condition.


\subsection{Non collapsed $\RCD(K,N)$ spaces}
Let us recall the definition of non collapsed $\RCD(K,N)$ m.m.s., as introduced in \cite{DePhilippisGigli18} (see also \cite{Kitabeppu17}, where Kitabeppu firstly investigated this class). 
\begin{definition}
	An $\RCD(K,N)$ metric measure space $(X,\dist,\meas)$ is said to be \textit{non collapsed} ($\ncRCD(K,N)$ for short) if $\meas=\mathcal{H}^N$, where $\mathcal{H}^N$ is the $N$-dimensional Hausdorff measure on $(X,\dist)$.
\end{definition}

Now we are ready to state the volume convergence theorem \cite[Theorem 1.2]{DePhilippisGigli18} in this setting and other definitions which will be useful for our aims.
\begin{theorem}[Volume convergence]\label{thm:volumeconvergence}
	Let $(X_n,\dist_n,\mathcal{H}^N,x_n)$ be a sequence of pointed $\ncRCD(K,N)$ m.m.s. with $K\in\mathbb{R}$ and $N\in [1,+\infty)$. Assume that $(X_n,\dist_n,x_n)$ converge in the pGH topology to $(X,\dist,x)$. Then precisely one of the following happens:
	\begin{itemize}
		\item[(a)] $\limsup_{n\to\infty}\mathcal{H}^N\left(B_1(x_n)\right)>0$. Then the $\limsup$ is a $\lim$ and $(X_n,\dist_n,\mathcal{H}^N,x_n)$ converge in the pmGH topology to $(X,\dist,\mathcal{H}^N,x)$;
		\item[(b)] $\lim_{n\to\infty}\mathcal{H}^N(B_1(x_n))=0$. In this case $\dim_{H}(X,\dist)\le N-1$. 
	\end{itemize}
\end{theorem}

\begin{definition}[Metric cone]\label{def:metriccone}
	Given a metric space $(Z,\dist_Z)$ we define the metric cone $C(Z)$ over $Z$ to be the completion of $\setR^+\times Z$ endowed with metric
	\begin{equation*}
	\dist_C\left((r_1,z_1),(r_2,z_2)\right)^2\doteq \begin{cases}
	r_1^2+r_2^2-2r_1r_2\cos\left(\dist_Z(z_1,z_2)\right) & \text{if}\quad \dist(z_1,z_2)\le\pi,\\
	\left(r_1+r_2\right)^2 & \text{if}\quad \dist(z_1,z_2)\ge \pi.
	\end{cases}
	\end{equation*}
\end{definition}

Thanks to the Bishop-Gromov inequality \eqref{prop:BishopGromovInequality}, the following definition can be given, following \cite{DePhilippisGigli18}.
\begin{definition}[Bishop-Gromov density]
	Given $K\in \mathbb{R}$, $N\in [1,+\infty)$ and an $\RCD(K,N)$ space $(X,\dist,\meas)$, 
	for any $x\in X$ we let the Bishop-Gromov density at $x$ be defined by
	\begin{equation}\label{BishopGromovDensity}
	\vartheta_N[X,\dist,\meas](x)\doteq \lim_{r\to 0} \frac{\meas(B_r(x))}{v_{K,N}(r)}.
	\end{equation}
\end{definition}

\begin{remark}\label{RemarkOnBishopGromov}
	We can define the Bishop Gromov density in \eqref{BishopGromovDensity} substituting $v_{K,N}(r)$ with $\omega_Nr^N$ at the denominator in \eqref{BishopGromovDensity} since
	\begin{equation}\label{LimitOfVKN}
	\lim_{r\to 0} \frac{\omega_N r^N}{v_{K,N}(r)} = 1.
	\end{equation}
\end{remark}

\begin{remark}\label{DensityLessThanOne}
	In \cite[Corollary 2.14]{DePhilippisGigli18} it is proved that if $(X,\dist,\mathcal{H}^N)$ is a $\ncRCD(K,N)$ space, with $K\in \mathbb{R}$ and $N\in [1,+\infty)$, then we have
	\begin{equation}
	\vartheta_N[X,\dist,\mathcal{H}^N](x) \leq 1 \qquad \text{for any } x\in X.
	\end{equation} 
	This follows from general results about differentiation of measures jointly with the lower semicontinuity of $\vartheta_N$, see \cite[Lemma 2.2]{DePhilippisGigli18}.
\end{remark}

\subsection{Almost volume cone implies almost metric cone}

It is possible to prove a rigidity result about Bishop-Gromov inequality in $\RCD(0,N)$ spaces which, roughly speaking, tells us that if we have equality of Bishop-Gromov ratios at two different radii then, at a certain scale, the space is isometric to a metric cone. This result is proven in \cite[Theorem 1.1, Theorem 4.1]{DePhilippisGigli16} in the case of $\RCD(0,N)$ and $\RCD(K,N)$ spaces respectively but we will state (part of) it here only in the case $K=0$.

\begin{theorem}[Volume cone implies metric cone]\label{FromVolumeConeToMetricCone}Let $N\in (0,+\infty)$ and $(X,\dist,\meas)$ be an $\RCD(0,N)$ space.
	Suppose there exist $x\in X$ and $R>r>0$ such that
	\begin{equation}
	\frac{\meas(B_R(x))}{R^n} = \frac{\meas(B_r(x))}{r^n}.
	\end{equation}
	Then, if the sphere $S_{\frac{R}{2}}(x)$ contains at least 3 points, we conclude that $N\geq 2$ and that there exists $\left(Z,\dist_z,\meas_z\right)$ an $RCD(N-2,N-1)$ space with $\diam Z \leq \pi$ such that the closed ball $\bar{B}_{\frac{R}{2}}(x)$ is isometric to the closed ball $\bar{B}_{\frac{R}{2}}(z)$ in the metric cone built over $Z$, where $z$ is the tip of the cone. This isometry sends $x$ to $z$.
	
	If the sphere $S_{\frac{R}{2}}(x)$ contains two points, then $\bar{B}_{\frac{R}{2}}(x)$ is isometric to $\left[-\frac{R}{2},\frac{R}{2}\right]$ with an isometry which sends $x$ to 0 while if the sphere $S_{\frac{R}{2}}(x)$ contains one point  $\bar{B}_{\frac{R}{2}}(x)$ is isometric to $\left[0,\frac{R}{2}\right]$ with an isometry which takes $x$ to 0.
\end{theorem}
\begin{remark}
	If $X$ is a Riemannian manifold with metric $g$, $\Ric \geq K$ and $\dim =N$, the existence of $x\in X$ and $R>r>0$ such that
	\begin{equation}\label{EqualityHere}
	\frac{\meas(B_R(x))}{v_{K,N}(R)} = \frac{\meas(B_r(x))}{v_{K,N}(r)}
	\end{equation}
	implies that the ball $B_R(x)$ with the Riemannian metric is isometric (in the Riemannian sense) to the ball $B_R(0)$ in the model with metric $g^{K,N}\doteq \de r^2 + s_{K,N}(r)^2g_{\mathbb{S}^{N-1}}$, where $s_{K,N}$ is defined in \eqref{DefinitionSkn} and $g_{\mathbb{S}^{N-1}}$ is the standard metric on $\mathbb{S}^{N-1}$. In the case $K=0$ the distance induced from the Riemannian metric $g^{0,N}=\de r^2+r^2g_{\mathbb{S}^{N-1}}$ is the cone distance introduced in \autoref{def:metriccone} if we take $\left(\mathbb{S}^{N-1},\dist_{g_{\mathbb{S}^{N-1}}}\right)$ as base space. It is important to note that in general this Riemannian isometry implies that the two balls are only locally isometric and the Riemannian isometry could  not extend to a metric isometry, which is the reason why in the statement of the previous theorem we have $\frac{R}{2}$ instead of $R$. 

	To see that in general the Riemannian isometry given by the rigidity in the Riemannian case does not extend to a global isometry, consider a cylinder in $\mathbb{R}^3$ with sections of diameter 1. Then take a point $x$ on it and a ball of radius $R=1$ centered at $x$. Even if \eqref{EqualityHere} holds with any $r<1$, $K=0$ and $N=2$ and the cylinder is a flat surface in $\mathbb{R}^3$ it is simple to see that the ball $\bar{B}_1(x)$ is not isometric to the euclidean ball $\bar{B}_1(0)$ in $\mathbb{R}^2$. 
\end{remark}
The previous rigidity theorem gives the possibility to deduce, arguing by compactness, an almost rigidity theorem. In fact Cheeger and Colding proved in \cite{CheegerColding96} a result of this flavour: if in a Riemannian manifold with a bound from below on Ricci curvature the Bishop-Gromov ratios at two radii $R$ and $r$ are almost equal, then the closed ball of radius $\frac{R}{2}$ in the manifold is close, in the sense of Gromov-Hausdorff distance, to the closed ball of radius $\frac{R}{2}$ around the tip of a suitably chosen metric cone. 

We can now rephrase and prove this result in the non smooth context, arguing by compactness and using \autoref{FromVolumeConeToMetricCone}.
\begin{theorem}[Almost volume cone implies almost metric cone - nc version]\label{FromAlmostVolumeCone}
	Given $\varepsilon>0$, $0<\eta<1$, $v>0$, $K\in \mathbb{R}$ and $N\in[2,+\infty)$, there exists $0<\delta\doteq \delta(K,N,\eta,v,\varepsilon)<1$ such that the following holds. If $\left(X,\dist,\mathcal{H}^N\right)$ is a $\ncRCD(K,N)$ space satisfying
	\begin{equation}
	\frac{\mathcal{H}^N(B_1(x))}{v_{K,N}(1)} \geq v >0,
	\end{equation}
	such that there exist $\delta>R>r>0$,
	with $\frac{r}{R}=\eta$, and $x\in X$ satisfying
	\begin{equation}
	\frac{\mathcal{H}^N(B_R(x))}{v_{K,N}(R)} \geq (1-\delta)\frac{\mathcal{H}^N(B_r(x))}{v_{K,N}(r)}, 
	\end{equation}
	then there exists $\left(Z,\dist_z,\meas_z\right)$ an $\RCD(N-2,N-1)$ space with $\diam Z \leq \pi$ such that, being $\bar{B}_{\frac{R}{2}}(z)$ the closed ball of radius $\frac{R}{2}$ around the tip $z$ of the metric cone built over $Z$, then
	\begin{equation}\label{Conclusion}
	\dist_{GH}\left(\bar{B}_{\frac{R}{2}}(x),\bar{B}_{\frac{R}{2}}(z)\right) \leq \varepsilon R.
	\end{equation}
\end{theorem}
\begin{proof}	
	Suppose by contradiction that there exist $\varepsilon_0$, a sequence $\delta_i\to 0$, $\left(X_i,\dist_i,\mathcal{H}^N_{\dist_i},x_i\right)$ which are $\ncRCD(K,N)$ spaces and $\delta_i>R_i>r_i>0$ radii with $\frac{r_i}{R_i}=\eta$ such that
	\begin{equation}\label{BishopReversed}
	\frac{\mathcal{H}^N_{\dist_i}(B_{R_i}(x_i))}{v_{K,N}(R_i)} \geq (1-\delta_i)\frac{\mathcal{H}^N_{\dist_i}(B_{r_i}(x_i))}{v_{K,N}(r_i)},
	\end{equation}
	and with
	\begin{equation}\label{GreaterThanEpsilon0}
	\dist_{GH}\left(\bar{B}_{\frac{R_i}{2}}(x_i),\bar{B}_{\frac{R_i}{2}}(z)\right) > \varepsilon_0R_i
	\end{equation}
	for each $\bar{B}_{\frac{R_i}{2}}(z)$ closed ball of radius $\frac{R_i}{2}$ around the tip $z$ of any metric cone built over $Z$, an $\RCD(N-2,N-1)$ space with $\diam Z\leq \pi$. If we suitably rescale the metric on these spaces
	\begin{equation*}
	\left(X_i',\dist_i',\meas_i',x_i'\right)\doteq \left(X_i,R_i^{-1}\dist_i,R_i^{-N}\mathcal{H}^N_{\dist_i},x_i\right),
	\end{equation*}
	then $(X_i',\dist_i',\meas_i',x_i')$ is a $\ncRCD(KR_i^2,N)$ space and $\meas_i'=\mathcal{H}^N_{\dist_i'}$. Read in these spaces \eqref{BishopReversed} becomes
	\begin{equation}\label{BishopReversed2}
	\frac{\mathcal{H}^N_{\dist_i'}(B_{1}'(x_i'))}{v_{K,N}(R_i)} \geq (1-\delta_i)\frac{\mathcal{H}^N_{\dist_i'}(B_{\eta}'(x_i'))}{v_{K,N}(r_i)},
	\end{equation}
	while \eqref{GreaterThanEpsilon0} tells us that
	\begin{equation}\label{GreaterThanEpsilon02}
	\dist_{GH}\left(\bar{B}'_{\frac{1}{2}}(x_i'),\bar{B}_{\frac{1}{2}}(z)\right) > \varepsilon_0
	\end{equation}
	for each $\bar{B}_{\frac{1}{2}}(z)$ closed ball of radius $\frac{1}{2}$ around the tip $z$ of any metric cone built over an $\RCD(N-2,N-1)$ space $Z$ with $\diam Z\leq \pi$. Here we tacitly exploited the fact that a metric cone is isometric to any rescaling of itself with center in the tip. 
	
	We also know that, there exist $C>0$ and $c>0$ depending only on $K$ and $N$ such that
	\begin{equation}\label{BoundOnVolume}
	C\geq\mathcal{H}^N_{\dist_i'}\left(B_1'(x_i')\right) = \frac{\mathcal{H}^N_{\dist_i}\left(B_{R_i}(x_i)\right)}{R_i^N}=\frac{\mathcal{H}^N_{\dist_i}\left(B_{R_i}(x_i)\right)}{v_{K,N}(R_i)}\frac{v_{K,N}(R_i)}{R_i^N}\geq vc
	\end{equation}
	because of the non collapsed condition \eqref{VolumeBound}, the bound on the density \eqref{DensityLessThanOne} and the fact that, for $0<R_i<1$, $\frac{v_{K,N}(R_i)}{R_i^N}$ is bounded uniformly from above and below. By compactness we have that, up to subsequences, 
	\begin{equation}
	\left(X_i',\dist_i',\mathcal{H}^N_{\dist_i'},x_i'\right) \to (X_{\infty},\dist_{\infty},\meas_{\infty},x_{\infty})
	\end{equation}
	where $(X_{\infty},\dist_{\infty},\meas_{\infty},x_{\infty})$ is a $\ncRCD(0,N)$ space as a consequence of the volume convergence (see \autoref{thm:volumeconvergence}) and \eqref{BoundOnVolume}. Passing to the limit \eqref{BishopReversed2}, taking into account that $\frac{v_{K,N}(R_i)}{v_{K,N}(r_i)}\to \frac{1}{\eta^N}$ and the Bishop-Gromov inequality in $X_{\infty}$  we obtain
	\begin{equation}
	\frac{\meas_{\infty}(B_1(x_{\infty}))}{\omega_N} = \frac{\meas_{\infty}(B_{\eta}(x_{\infty}))}{\omega_N\eta^N}.
	\end{equation}
	Since $\dim_H\left( X_{\infty} \right)=N>1$ we can exclude the degeneracy cases in \autoref{FromVolumeConeToMetricCone}, thus we obtain the existence of $\left(Z,\dist_z,\meas_z\right)$ an $\RCD(N-2,N-1)$ space with $\diam Z\leq \pi$ such that $\bar{B}_{\frac{1}{2}}(x_{\infty})$ is isometric to the closed ball $\bar{B}_{\frac{1}{2}}(z)$ in the metric cone built over $Z$, where $z$ is the tip of the cone. Then
	\begin{equation}
	\dist_{GH}\left(\bar{B}_{\frac{1}{2}}(x_{\infty}),\bar{B}_{\frac{1}{2}}(z)\right)=0
	\end{equation}
	which contradicts \eqref{GreaterThanEpsilon02}.
\end{proof}
\begin{remark}
	With the same proof, when we work in the class of $\ncRCD(0,N)$ spaces, we obtain the same statement as before with the constraint $R<1$ instead of $R<\delta$.
\end{remark}

\subsection{Almost cone splitting}\label{subsec:almostcone splitting}

\begin{definition}\label{def:GHequivalence}
	Given metric spaces $(X,\dist_X)$ and $(Y,\dist_Y)$, we say that $f:X\to Y$ is an $\varepsilon$-GH equivalence if
	$|\dist_Y(f(x_1),f(x_2))-\dist_X(x_1,x_2)|\leq \varepsilon$ for all $x_1,x_2\in X$, and for all $y\in Y$ there exists $x\in X$ such that $\dist_Y(f(x),y) \leq \varepsilon$.
\end{definition}

\begin{definition}\label{def:Conicality}
	Given a metric space $(X,\dist)$, we define the $t$-conicality of the ball $B_r(x)$ as
	\begin{equation}
	\mathcal{N}_t(B_r(x)) \doteq \inf_{\varepsilon > 0} \left\{ \exists \ Z \  \mbox{metric cone  and $\RCD(0,N)$ with tip $z$} \ \mbox{s.t.} \ \dist_{GH}\left(\bar{B}_{\frac{tr}{2}}(x),\bar{B}_{\frac{tr}{2}}(z)\right) \leq \frac{\varepsilon r }{2} \right\}.
	\end{equation}
\end{definition}

\begin{definition}\label{def:ConicalSet}
	Following \cite{CheegerNaber13a} we define the \emph{ $\varepsilon-(t,r)$ conical set} in $B_{\frac{1}{2}}(x_0)$ as
	\begin{equation}
	C^{\varepsilon}_{t,r} \doteq \{x\in B_{\frac{1}{2}}(x_0) : \mathcal{N}_t(B_r(x))< \eps\},
	\end{equation}
	where $\mathcal{N}$ is defined in \eqref{def:Conicality}. 
\end{definition}

\begin{theorem}[Cone splitting, quantitative version]\label{thm:conesplittingquant}
	For all $K\in \mathbb{R}$, $N\in [2,+\infty)$, $0<\gamma<1$, $\delta<\gamma^{-1}$, and for all $\tau,\psi>0$ there exist $0<\varepsilon(N,K,\gamma,\delta,\tau,\psi)<\psi$ and $0<\theta=\theta(N,K,\gamma,\delta,\tau,\psi)$ such that the following holds. Let $(X,\dist,\meas)$ be an $\RCD(K,N)$ m.m.s., $x\in X$ and $r\le\theta$ be such that there exists an $\varepsilon r$-GH equivalence
	\begin{equation*}
	F:B_{\gamma^{-1}r}\left((0,z^*)\right)\to B_{\gamma^{-1}r}(x)
	\end{equation*}
	for some cone $\setR^l\times C(Z)$, with $(Z,\dist_Z,\meas_Z)$ an $\RCD(N-l-2,N-l-1)$ m.m.s..
	If there exists 
	\begin{equation*}
	x'\in B_{\delta r}(x)\cap C^{\varepsilon}_{\gamma^{-N},\delta r}
	\end{equation*}
	with 
	\begin{equation*}
	x'\notin T_{\tau r}\left(F(\setR^l\times\left\lbrace z^*\right\rbrace )\right)\cap B_r(x),
	\end{equation*}
	where $T_{\tau r}(\cdot)$ is the tubular neighbourhood of radius $\tau r$, then for some cone $\setR^{l+1}\times C(\tilde{Z})$, where  $(\tilde{Z},\dist_{\tilde{Z}},\meas_{\tilde{Z}})$ is a $\RCD(N-l-3,N-l-2)$ m.m.s.,
	\begin{equation*}
	\dist_{GH}\left(B_r(x), B_r((0,\tilde{z}^*))\right)<\psi r.
	\end{equation*}  
\end{theorem}

\autoref{thm:conesplittingquant} is a quantitative version of the following statement: if a metric cone with vertex $z$ is a metric cone also with respect to $z'\neq z$, then it contains a line. 
It can be rigorously stated in the setting of $\RCD$ spaces as follows.


\begin{proposition}[Cone splitting, rigid version]\label{prop:conesplitting}
	Let $(X,\dist,\meas)$ be an $\RCD(0,N)$ m.m.s. isomorphic to $\setR^l\times C(\bar{Z})$ for some $l\in\setN$ and some $\RCD(N-l-2,N-l-1)$ m.m.s. $(Z,\dist_Z,\meas_Z)$. Let $\bar{z}$ be the vertex of $C(\bar{Z})$ and suppose that there exist a metric cone  $C(\hat{Z})$ with vertex $\hat{z}$ and an isometry $I:C(\hat{Z})\to X$ such that $I(\hat{z})\notin \setR^l\times\left\lbrace \bar{z}\right\rbrace$. Then $(X,\dist,\meas)$ is isomorphic to $\setR^{l+1}\times C(\tilde{Z})$ for some $\RCD(N-l-3,N-l-2)$ m.m.s $(\tilde{Z},\dist_{\tilde{Z}},\meas_{\tilde{Z}})$.
\end{proposition}

\begin{proof}
	The sought conclusion can be achieved through two intermediate steps.
	
	\textbf{Step 1.} Aim of this first step is to prove that $(X,\dist,\meas)$ contains a line passing through $I(\hat{z})$ (and therefore with non trivial component on the $C(\bar{Z})$ factor). In order to do so we wish to prove that the ray connecting $(0,\bar{z})$ to $I(\hat{z})$ actually extends to a line. Indeed, taking into account the fact that locally around $I(\hat{z})$ it is a geodesic we obtain that the cross section $\hat{Z}$ contains two points $z_1,z_2$ such that $\dist_{\hat{Z}}(z_1,z_2)=\pi$. Hence, considering now the ray emanating from $I(\hat{z})$ and passing through $(0,\bar{z})$, which corresponds to the point $z_2$ without loss of generality, we obtain that also $\bar{Z}$ contains points $\bar{z}_1,\bar{z}_2$ such that $\dist_{\bar{Z}}(\bar{z}_1,\bar{z}_2)=\pi$, otherwise the ray above would not be minimizing around $(0,\bar{z})$. Hence, as we claimed, there is a line in $X$ passing through $I(\hat{z})$ and $(0,\bar{z})$.   
	
	\textbf{Step 2.}
	The sought conclusion about the additional splitting follows from what we proved in \textbf{Step 1} applying \autoref{lemma:additionalsplitting} below. To conclude it suffices to observe that the split factor is still a metric cone since the whole space is. Indeed if a product $\mathbb{R}^l\times Z$ is a metric cone, then it can be viewed as a metric cone on his sphere of radius $1$ and $Z$ is the cone over the intersection of this sphere with the section $\{0\}\times Z$. 
\end{proof}

\begin{remark}\label{rm:conesplitweaker}
	Let us remark that the same conclusion of \autoref{prop:conesplitting} above holds true under the following weaker assumption: with the same notation adopted above, there exist $z\notin\setR^l\times\{\bar{z}\}$ and an isometry $I:B_r(\hat{z})\to B_r(z)$ for some $r>\dist(z,(0,\bar{z}))$ such that $I(\hat{z})=z$. This stronger statement can be checked with no modification w.r.t. the proof we presented above. 
\end{remark}

\begin{lemma}[Additional splitting]\label{lemma:additionalsplitting}
	Let $(X,\dist,\meas)$ be an $\RCD(0,N)$ m.m.s. isomorphic to $\setR^l\times Y$ for some $l\in\setN$ and some $\RCD(0,N-l)$ m.m.s. $(Y,\dist_Y,\meas_Y)$. Suppose that $X$ contains a line whose component on the factor $Y$ is non constant. Then $X$ is isomorphic to $\setR^{l+1}\times Z$ for some $\RCD(0,N-l-1)$ m.m.s. $(Z,\dist_Z,\meas_Z)$.
\end{lemma}

\begin{proof}
	We just briefly outline the strategy of the proof.\\
	Let us begin by observing that in the case $l=0$ the statement corresponds to the splitting theorem, proved in this generality in \cite{Gigli13}.\\
	If $l\ge1$ we wish to prove that the existence of a line with non constant $Y$-component implies the existence of a splitting function on the $Y$-factor and therefore the conclusion. In order to do so, first we build the Busemann function $u$ associated to the given line. From \cite{Gigli13} we know that $\Delta u=0$ and $|\nabla u|=1$ and then from the Bochner formula, $\Hess u=0$ (see \cite{Gigli18} and \cite{Han18b}). Let us denote furthermore by $x_1,\dots,x_l$ the coordinate functions of the Euclidean factor $\setR^l$. We claim that there exist real numbers $a_1,\dots,a_l$ such that $f\doteq u-a_1x_1\cdots-a_lx_l$ is a non constant function with constant minimal upper gradient, independent of the Euclidean variable and with vanishing Hessian. 
	Indeed 
	we can define $a_i\doteq \nabla u \cdot \nabla x_i$. These numbers are constant because of the fact that $\Hess u=0$ and $\Hess x_i=0$. Then taking $f$ as before, from $\Delta u=0$ it follows $\Delta f=0$ and from $\Hess u=0$ it follows $\Hess f = 0$; while from the fact that $|\nabla u|=1$, it follows, because of the tensorization, that $|\nabla f|$ is constant. Such a function induces a splitting function on $(Y,\dist_Y,\meas_Y)$ and the sought conclusion can be obtained applying the results in the appendix of \cite{Han18}, which inspires the fortchoming \autoref{lem:Splitting}. One has to verify that $f$ is not constant: if not $u$ will be an affine function on $\mathbb{R}^l$ and then the Busemann function of a line entirely contained in the first factor $\mathbb{R}^l$. This is not possible since the line associated to the Busemann function $u$ had a non-trivial projection over the second factor, while in this case $u$ would be the Busemann function associated to a line in the factor $\mathbb{R}^l$.
\end{proof}

\begin{lemma}[Functional splitting]\label{lem:Splitting}
	Let $(X,\dist,\meas)$ be an $\RCD(0,N)$ space and let $u:X\to \mathbb R$ a function such that $\Delta u=0$ and $|\nabla u|=1$. Then $(X,\dist,\meas)$ is isomorphic to $(Y\times \mathbb{R},\dist_Y\times \dist_{eucl},\meas_Y\times \leb^1)$.  
	
	Analogously, if there exist functions $u_1,\dots,u_l:X\to \mathbb R$ such that for all $i=1,\dots,l$ it holds $\Delta u_i=0$ and $|\nabla u_i|=1$ in $(X,\dist,\meas)$ as before, and $\nabla u_i \cdot \nabla u_j=0$ for all $1\leq i<j\leq l$, then $(X,\dist,\meas)$ is isomorphic to $(Y\times \mathbb{R}^l,\dist_Y\times \dist_{eucl},\meas_Y\times \leb^l)$.
\end{lemma}
\begin{proof}
     From the improved Bochner formula \cite[Corollary 3.3.9]{Gigli18} it follows that $\Hess u=0$. Then one can consider the regular Lagrangian flow $X_t$ (see \cite{AmbrosioTrevisan14} for the definition of regular Lagrangian flow) associated to $\nabla u$.     
     Since $\Delta u=0$ and $\Hess u=0$ we can use \cite[Theorem 1.9, (iv)]{AmbrosioBrueSemola18} to deduce that for every $x,y \in X$,
	\begin{equation}\label{Split1}
	\dist(\XX_t(x),\XX_t(y))=\dist(x,y) \qquad \forall t>0.
	\end{equation}
	Then, since for every $x\in X$
	$$
	\frac{\mbox{d}}{\mbox{d}t} u(\XX_t(x))=\nabla u\cdot \nabla u(\XX_t(x))=1,
	$$
	it follows that for $x\in X$, $u(\boldmath{X}_t(x))-u(x)=t$. Using this information, jointly with the fact that $u$ has a $1$-Lip representative, being $|\nabla u|=1$, and the fact that $\dist(\XX_{t}(x),x)\leq t$ because of $|\nabla u|\leq 1$, it follows that for every $x\in X$ and $t>0$,
	\begin{equation}\label{Split2}
	\dist(\XX_t(x),x)=t.
	\end{equation}
	From \eqref{Split1} and \eqref{Split2} it follows the splitting as in \cite[Sections 5 to 7]{Gigli13} up to substituting $\XX_t(x)$ with $F_t(x)$ therein.\\	
	For the multiple splitting, one argues precisely as in \cite[Conclusion of Theorem 5.1]{MondinoNaber14}.

\end{proof}

\begin{proof}[Proof of \autoref{thm:conesplittingquant}]
	The conclusion follows from \autoref{prop:conesplitting} via rescaling and a compactness argument.\\
	Let us suppose by contradiction that the statement is not satisfied. After rescaling we obtain the existence of sequences $\varepsilon_n\downarrow 0$ and $\alpha_n\downarrow 0$, of a sequence of $\RCD(-\alpha_n,N)$ m.m.s. $(X_n,\dist_n,\meas_n)$, of points $x_n\in X_n$ and of $\varepsilon_n$-GH equivalences
	\begin{equation*}
	F_n:B_{\gamma^{-1}}((0,z_n^*))\to B_{\gamma^{-1}}(x_n),
	\end{equation*}
	where $(0,z_n^*)$ denotes the vertex of a cone $\setR^l\times C(Z_n)$. Furthermore there are points 
	\begin{equation*}
	x_n'\in B_{\delta}(x_n)\cap C^{\varepsilon_n}_{\gamma^{-N},\delta}
	\end{equation*}
	with 
	\begin{equation*}
	x_n'\notin T_{\tau}\left(F_n(\setR^l\times\left\lbrace z_n^*\right\rbrace )\right)\cap B_{1}(x_n)
	\end{equation*}
	and the estimate
	\begin{equation}\label{eq:contr}
	\dist_{GH}\left(B_1(x_n),B_1((0,\tilde{z}^*))\right)\ge\psi
	\end{equation}
	is satisfied for any cone of the form $\setR^{l+1}\times C(\tilde{Z})$, where  $(\tilde{Z},\dist_{\tilde{Z}},\meas_{\tilde{Z}})$ is a $\RCD(N-l-3,N-l-2)$ metric measure space.
	Passing to the limit all the conditions above, by compactness and stability (see \autoref{remark:stability}) we obtain an $\RCD(0,N)$ m.m.s. $(X,\dist,\meas)$, $x\in X$, $l\in\setN$, an $\RCD(N-l-2,N-l-1)$ m.m.s. $(Z,\dist_Z,\meas_Z)$ and an isometry
	\begin{equation}\label{eq:lociso}
	F:B_{\gamma^{-1}}((0,z^*))\to B_{\gamma^{-1}}(x),
	\end{equation}
	where $(0,z^*)$ is a vertex of the cone $\setR^l\times C(Z)$.
	Furthermore we can find $x'\in B_{\delta}(x)$ such that $B_{2^{-1}\delta\gamma^{-N}}(x')$\footnote{Note that $2^{-1}\delta\gamma^{-N}>\delta$ which will be important to end the proof by applying  a localized version of \autoref{prop:conesplitting} around $x$ and $(0,z^*)$ (see also \autoref{rm:conesplitweaker}).} is isometric to the ball centred in the tip of a metric cone and 
	\begin{equation}\label{eq:confar}
	x'\notin T_{\tau}\left(F(\setR^l\times C(Z))\right)\cap B_1(x) 
	\end{equation}
	and, by \eqref{eq:contr}, we get that 
	\begin{equation}\label{eq:farfromcone}
	\dist_{GH}\left(B_1(x),B_1((0,\tilde{z}^*))\right)\ge\psi,
	\end{equation}
	for any cone of the form $\setR^{l+1}\times C(\tilde{Z})$, where  $(\tilde{Z},\dist_{\tilde{Z}},\meas_{\tilde{Z}})$ is an $\RCD(N-l-3,N-l-2)$ metric measure space.\\
	Taking into account a localized version of \autoref{prop:conesplitting} around $x$ and $(0,z^*)$ (see also \autoref{rm:conesplitweaker}), the combination of \eqref{eq:lociso}, \eqref{eq:confar} and \eqref{eq:farfromcone} gives the sought contradiction.
\end{proof}

\subsection{Singular sets on noncollapsed $\RCD(K,N)$ spaces}
In this subsection we briefly review the main structural results for non collapsed $\RCD(K,N)$ spaces.

Given a m.m.s. $(X,\dist,\meas)$, $x\in X$ and $r\in(0,1)$, we consider the rescaled and normalized pointed m.m.s. $(X,r^{-1}\dist,\meas_r^{x},x)$, where 
\begin{equation*}
\meas_r^x\doteq \left( \int_{B(x,r)} 1-\frac{\dist(x,y)}{r} \di \meas(y)\right)^{-1}\meas.
\end{equation*}

\begin{definition}
	Let $(X,\dist,\meas)$ be an $\RCD(K,N)$ m.m.s. for some $1<N<+\infty$, $K\in\setR$ and let $x\in X$. We say that a pointed m.m.s. $(Y,\dist_Y,\eta,y)$ is tangent to $(X,\dist,\meas)$ at $x$ if there exists a sequence $r_i\downarrow 0$ such that $(X,r_i^{-1}\dist,\meas_{r_i}^x,x)\rightarrow(Y,\dist_Y,\eta,y)$ in the pmGH topology. The collection of all the tangent spaces of $(X,\dist,\meas)$ at $x$ is denoted by $\Tan(X,\dist,\meas,x)$.
\end{definition}

A compactness argument, which is due to Gromov, together with the rescaling and stability properties of the $\RCD(K,N)$ condition (see \autoref{remark:stability}), yields that $\text{Tan}(X,\dist,\meas,x)$ is non empty for every $x\in X$ and its elements are all $\RCD(0,N)$ pointed m.m.s..

In the special case in which $(X,\dist,\meas)$ is non collapsed any tangent cone has a conical structure, we refer to \cite{DePhilippisGigli18} for the proof of this result.

\begin{theorem}\label{thm:tangentconesaremetricocnes}
Let $(X,\dist,\meas)$ be a $\ncRCD(K,N)$ metric measure space. Then, for any $x\in X$, any $(Y,\dist_Y,\eta,y)\in\Tan(X,\dist,\meas,x)$ is a metric cone according to \autoref{def:ConicalSet}.
\end{theorem}
As a consequence of the structural property proved in \cite{MondinoNaber14} it is simple to see that if $(X,\dist,\meas)$ is a $\ncRCD(K,N)$ m.m.s. then $N$ is integer and the \textit{regular set}
\begin{equation}
	\mathcal{R}\doteq
	\left\lbrace x\in X:\ \Tan(X,\dist,\meas,x)=\left\lbrace(\setR^N,\dist_{eucl},c_N\leb^N,0^N)\right\rbrace \right\rbrace,
	\quad c_N\doteq \frac{N+1}{\omega_N}
\end{equation}
satisfies $\meas(X\setminus \mathcal{R})=0$. The \textit{singular set} of $X$ is the complement of $\mathcal{R}$. In \cite{CheegerColding97} Cheeger and Colding, inspired by the stratification results of geometric measure theory, introduced a way to stratify the singular set of a non collapsed Ricci limit according to the maximal dimension of the Euclidean factor split off by a tangent space.
This definition can be given also in the context of $\ncRCD(K,N)$ spaces and reads as follows:

\begin{definition}\label{def:stratification}
Let $(X,\dist,\meas)$ be a $\ncRCD(K,N)$ m.m.s..
Given $x\in X$ and $0\le k\le N$ we say that $x\in\mathcal{S}^k$ if no tangent space of $(X,\dist,\meas)$ at $x$ splits off isometrically a factor $\setR^{k+1}$.
\end{definition}
Note that we have the inclusions 
\begin{equation*}
	\mathcal{S}^0\subset \mathcal{S}^1\subset ...\subset \mathcal{S}\doteq X\setminus \mathcal{R}.
\end{equation*}

\begin{example}
Let $K$ be the region delimited by a triangle in $\setR^2$. Let $l_i$ be the edges of the triangle and $v_i$ be its vertexes. Then $(K,\dist_{eucl},\leb^2\res K)$\footnote{With $\leb^2\res K$ we mean $\leb^2\res K(E)\doteq \leb^2(K\cap E)$ for every $E$ Borel.} is a non collapsed $\RCD(0,2)$ m.m.s. (it is not a non collapsed Ricci limit of a sequence of two dimensional Riemannian manifolds instead, as it follows from \cite{CheegerColding2000a}). Observe that all the points in the interior of $K$ are regular points. The interior points of the edges belong to $\mathcal{S}^1\setminus\mathcal{S}^0$, while the vertexes are in $\mathcal{S}^0$. 
\end{example}

\begin{theorem}\label{thm:dimensionestimate}
Let $(X,\dist,\meas)$ be a $\ncRCD(K,N)$ m.m.s. with $K\in\mathbb{R}$ and $N\in[1,+\infty)$. Then it holds that $\dim_{H}\mathcal{S}^k\le k$ for any $0\le k\le N$.
\end{theorem}

\begin{proof}
We refer to \cite[Theorem 4.7]{CheegerColding97} for the proof of this result for non collapsed Ricci limits and to \cite[Theorem 1.8]{DePhilippisGigli18} for its generalization to $\ncRCD$ spaces.\\
Let us just recall here that the proof is based on a dimension reduction argument and on the use of the splitting theorem \cite{Gigli13}, together with \autoref{thm:tangentconesaremetricocnes}.
\end{proof}

\begin{remark}
	It is possible to find examples of non collapsed Ricci limit spaces of dimension $2$ such that $\mathcal{S}^0$ is dense (see for instance \cite[Subsection 3.4]{CheegerJangNaber18}). Hence, in general, $\mathcal{H}^k$ is not locally finite when restricted to $\mathcal{S}^k$.
\end{remark}

\section{Volume bound for the quantitative strata}\label{sec:mainresult}

\subsection{Statement and basic consequences}
A quantitative counterpart of the stratification in \autoref{def:stratification} was introduced in \cite{CheegerNaber13a} in the setting of non collapsed Ricci limit spaces. The definition extends to the case of $\ncRCD$ spaces with no modification.

\begin{definition}
For any $\eta>0$ and any $0<r<1$, define the $k^{th}$-effective stratum $\mathcal{S}^k_{\eta,r}$ by
\begin{equation*}
\mathcal{S}^{k}_{\eta,r}\doteq \set{y|\dist_{GH}(B_s(y),B_s\left((0,z^*)\right))\ge\eta s\quad\text{for all } \setR^{k+1}\times C(Z)\quad\text{and all } r\le s\le1},
\end{equation*}
where $B_s\left((0,z^*)\right)$ denotes the ball in $\setR^{k+1}\times C(Z)$ centered at $(0,z^*)$ with radius $s$.
\end{definition}
Since it plays a role in the sequel of the note, we point out here that, given metric spaces $(X,\dist_X)$ and $(Y,\dist_Y)$, the notions ``$\dist_{GH}(X,Y)\le\varepsilon$'' and ``there exists an $\varepsilon$-GH isometry between $X$ and $Y$'' are only equivalent up to a multiplicative constant which, however plays no role for the sake of our discussion. We refer to \cite[Chapter 27]{Villani09} for more details about this point.

Let us observe now that
\begin{equation}\label{eq:inclusion singular set}
\mathcal{S}^k_{\eta,r}\subset\mathcal{S}^{k'}_{\eta',r'},\quad \text{if  $k\le k'$, $\eta'\le\eta$ and $r\le r'$}
\end{equation}
and 
\begin{equation}\label{eq:sngular vs quantitative singular}
\mathcal{S}^k=\bigcup_{\eta}\bigcap_{r}\mathcal{S}^{k}_{\eta,r}.
\end{equation}
Indeed, if $y\in \mathcal{S}^k$ then $y\in \bigcap_{r}\mathcal{S}^{k}_{\eta,r}$ for some $\eta>0$ and it is trivial to see that $\bigcap_{r}\mathcal{S}^{k}_{\eta,r}\subset \mathcal{S}^k$.

The classical stratification is built separating points according to the infinitesimal symmetries of the space. The quantitative stratification instead is based on how many symmetries there are on balls of a definite size at any point.
 
\begin{remark}\label{remark: singular vs quantitative singular}
	In \eqref{eq:sngular vs quantitative singular} we can consider just the union over $0<\eta<\eps$ for some $\eps>0$ fixed, or even over a countable sequence $\eta_i\to 0$.
\end{remark}
\begin{remark}
Let us remark that on a smooth Riemannian manifold the strata $\mathcal{S}^k$ are all empty, instead the effective strata $\mathcal{S}^{k}_{\eta,r}$ are non trivial.
\end{remark}

Let us state the main result of this note, which extends to the synthetic framework the result proved for non collapsed Ricci limit spaces in \cite{CheegerNaber13a}. As we already pointed out in the introduction, this statement has already been useful, very recently, in the proof of \cite[Theorem 5.8]{MondinoKapovitch19}, dealing with stability properties for the boundary of non collapsed $\RCD(K,N)$ spaces.
\begin{theorem}\label{thm:RCDquantitativevolumebound}
	Given $K\in\mathbb{R}$, $N\in [2,+\infty)$, $k\in [1,N)$ and $v,\eta>0$, there exists a constant $c(K,N,v,\eta)>0$ such that if $(X,\dist,\mathcal{H}^N)$ is a $\ncRCD(K,N)$ m.m.s. satisfying
	\begin{equation}\label{VolumeBound}
		\frac{\mathcal{H}^N(B_1(x))}{v_{K,N}(1)}\ge v
		\qquad
		\forall x\in X,
	\end{equation}
	then, for all $x\in X$ and  $0<r<1/2$, it holds	
	\begin{equation}\label{eq:volumebound}
	\mathcal{H}^N(\mathcal{S}^{k}_{\eta,r}\cap B_{1/2}(x))\le c(K,N,v,\eta)r^{N-k-\eta}.
	\end{equation}
\end{theorem}

Let us make a few remarks about \eqref{eq:volumebound}. First we wish to prove that it implies the standard Hausdorff dimension estimate $\dim_{H}(\mathcal{S}^k)\le k$.
To do so let us observe that the $\eta r$-enlargement of $\mathcal{S}^k_{2\eta, r}$ is a subset of $\mathcal{S}^k_{\eta,r}$, that is to say
\begin{equation}\label{z3}
T_{\eta r}(\mathcal{S}^k_{2\eta,r})\doteq \set{x\in X:\ \dist(x,\mathcal{S}^k_{2\eta,r})<\eta r}\subset \mathcal{S}^k_{\eta,r}.
\end{equation}
To check \eqref{z3} it is enough to use the triangle inequality: take $x\in T_{\eta r}(\mathcal{S}^k_{2\eta,r})$, by definition there exists $x'\in \mathcal{S}^k_{2\eta,r}$ such that $\dist(x,x')< \eta r$, hence we have
\begin{align*}
	\dist_{GH}(B_s(x),B_s((0,z^*)))&\ge \dist_{GH}(B_s(x'),B_s((0,z^*)))-\dist_{GH}(B_s(x'),B_s(x))\\
	&\ge 2\eta s-\eta s\\
	&=\eta s
\end{align*}
for any $\setR^{k+1}\times C(Z)$ with $z^*$ tip of $C(Z)$ and every $r\le s\le 1$, where in the last inequality we used $x\in \mathcal{S}^k_{2\eta,r}$ and $\dist_{GH}(B_s(x'),B_s(x))\le \dist(x,x')<\eta r\le \eta s$.
With $\eqref{z3}$ at our disposal we can strengthen \eqref{eq:volumebound} obtaining a volume estimate of the $\eta r$-enlargement of the quantitative strata
\begin{equation}\label{eq:22}
	\mathcal{H}^N\left(T_{\eta r}(\mathcal{S}^k_{2\eta,r})\cap B_{1/2}(x)\right)\le c r^{N-k-\eta}
	\qquad
	\forall x\in X.
\end{equation}
In particular, \eqref{eq:22} implies that
\begin{equation}\label{z60}
\mathcal{H}^N\left(T_{\eta r}\left(\bigcap_{s>0}\mathcal{S}^k_{2\eta,s}\right)\cap B_{1/2}(x)\right)\le c r^{N-k-\eta}
\qquad
\text{for any}\  0<r\le 1/2,\quad \forall x\in X,
\end{equation}
that, together with a localized version of \autoref{lemma: tubolar estimate} below, gives
\begin{equation}\label{z4}
	\mathcal{H}^{k+\eta+\eps}\left( \bigcap_{s>0}\mathcal{S}^k_{2\eta,s} \right) =0
	\qquad
	\forall \eps>0.
\end{equation}
Recalling that $\mathcal{S}^k=\bigcup_{\eta>0}\bigcap_{0<s<\eps}\mathcal{S}^{k}_{\eta,s}$ for any $\eps>0$ and that the union in $\eta$ can be taken countable (see \autoref{remark: singular vs quantitative singular}) we get eventually $\dim_{H}(\mathcal{S}^k)\le k$.

\begin{lemma}\label{lemma: tubolar estimate}
	Let $(X,\dist,\mathcal{H}^N)$ be a $\ncRCD(K,N)$ m.m.s. satisfying \eqref{VolumeBound} and let $E\subset X$ be Borel. If for some $0<\alpha\le N$ it holds that
	\begin{equation*}
		\mathcal{H}^N(T_{r_i}(E))\le Cr_i^{N-\alpha}
		\quad \text{for a sequence}\ r_i\downarrow 0,
	\end{equation*}
	then
	\begin{equation*}
		\mathcal{H}^{\alpha}(E)\le c(K,N,\alpha,v)C.
	\end{equation*}
\end{lemma}

\begin{proof}
	Let us fix $0<r_i<1$ and $\delta \geq 10 r_i$.
	By means of a standard covering theorem (see \cite[Theorem 1.2]{Heinonen01}) we can find a finite family of points $x_1,...,x_m$ (a priori a \emph{countable} family, but finite if we take into account the estimate \eqref{z5} below) in $E$ such that $\set{B_{r_i}(x_k)}_{k=1,...,m}$ is disjoint and $E\subset \bigcup_{k=1}^m B_{5r_i}(x_k)$. Let us estimate $m$. From the inclusion $\bigcup_{k=1}^mB_{r_i}(x_k)\subset T_{r_i}(E)$ and the fact that $\set{B_{r_i}(x_k)}_{k=1,...,m}$ is a disjoint family we deduce
	\begin{equation*}
		\sum_{k=1}^m \mathcal{H}^N(B_{r_i}(x_k))\le 	\mathcal{H}^N(T_{r_i}(E))\le Cr_i^{N-\alpha}.
	\end{equation*}
	On the other hand the Bishop-Gromov inequality and \eqref{VolumeBound} grant
	\begin{equation*}
		\mathcal{H}^N(B_{r_i}(x_k))\geq v_{K,N}(r_i) \frac{\mathcal{H}^N(B_1(x_k))}{v_{K,N}(1)}
		\ge c r_i^N v,
	\end{equation*}
	where $c>0$ depends only on $K$ and $N$.
	Thus 
	\begin{equation}\label{z5}
		m\le \frac{C}{cv} r_i^{-\alpha}.
	\end{equation}
	Since $E\subset \bigcup_{k=1}^m B_{5r_i}(x_k)$ and $\delta > 10 r_i$, we get
	\begin{equation*}
		\mathcal{H}_{\delta}^{\alpha}(E)\le c(\alpha) \sum_{i=1}^m (\text{diam}(B_{5r_i}(x_k)))^{\alpha}\le mc'r_i^{\alpha} \le \frac{Cc'}{cv},
	\end{equation*}
	where $c'>0$ depends only on $\alpha$ and we used \eqref{z5} in the last passage. Letting $\delta\to 0$ we obtain the sought conclusion.
\end{proof}

Let us also mention that, even though \eqref{eq:volumebound} is stronger than $\dim_{H}(\mathcal{S}^k)\le k$ it does not imply
\begin{equation*}
	\mathcal{H}^k\left(\bigcap_{r>0}\mathcal{S}^k_{\eta,r}\cap B_{1/2}(x)\right)<\infty,
\end{equation*}
one of the problems being the term $r^{\eta}$ appearing at the right hand side of \eqref{eq:volumebound}. An improvement in this direction is one of the fundamental results in \cite{CheegerJangNaber18}.

\subsubsection{Estimate for the $r$-enlargement of the boundary}
In \cite{DePhilippisGigli18} the authors have proposed a definition of boundary $\partial X$ of a $\ncRCD(K,N)$ m.m.s. $X$ as
\begin{equation*}
	\partial X\doteq\text{closure of } \mathcal{S}^{N-1}\setminus \mathcal{S}^{N-2}.
\end{equation*}
We can use \autoref{thm:RCDquantitativevolumebound} to estimate the measure of the $r$-enlargement of $\partial X$.
\begin{corollary}
	Given $K\in\mathbb{R}$, $N\in [2,+\infty)$ and $v, \eta >0$, there exist $c(K,N,v,\eta)>0$ and $r(K,N)>0$ such that, if $(X,\dist,\mathcal{H}^N)$ is a $\ncRCD(K,N)$ m.m.s. satisfying \eqref{VolumeBound}, then, for all $x\in X$ and $0<r<r(K,N)$, it holds	
	\begin{equation*}
	\mathcal{H}^N(T_r(\partial X)\cap B_{1/2}(x))\le c(K,N,v,\eta)r^{1
		-\eta}.
	\end{equation*}
\end{corollary}
\begin{proof}
	Let us denote by $k>0$ the biggest constant such that
	\begin{equation}\label{z61}
		\frac{v_{-s,N}(1)}{\omega_N}\le \frac{3}{2}
		\quad \text{for any } 0<s<k.
	\end{equation}
	Note that $k$ depends only on $N$. The proof is divided in two steps.\\
	\textbf{Step1.} Aim of this first step is to prove our conclusion under the additional assumption $K>-k$.\\	
	Let us first observe that, for any $z\in \mathcal{S}^{N-1}\setminus \mathcal{S}^{N-2}$, the Euclidean half space of dimension $N$ belongs to $\Tan_z(X,\dist,\mathcal{H}^N)$. To check this statement we build upon three ingredients. The first is that, by the very definition of the singular strata, $\Tan_z(X,\dist,\mathcal{H}^N)$ must contain a $\ncRCD(0,N)$ m.m.s. that splits off $\setR^{N-1}$ but not $\setR^N$. The second ingredient is the characterization of $\RCD(0,1)$ spaces provided in \cite{KitabeppuLakzian16} and the last one is the fact that tangent cones are metric cones (see \autoref{thm:tangentconesaremetricocnes}).
	
	Let $z\in \mathcal{S}^{N-1}\setminus \mathcal{S}^{N-2}$ be fixed. 
	Applying \cite[Theorem 1.3]{DePhilippisGigli18} we get that $\vartheta_N[X,\dist,\mathcal{H}^N](z)=1/2$ (see \eqref{BishopGromovDensity} for the definition of $\vartheta_N$). Thus, as a consequence of \eqref{prop:BishopGromovInequality} and \eqref{z61}, we have
	\begin{equation*}
		\frac{\mathcal{H}^N(B_r(z))}{\omega_Nr^N}=\frac{\mathcal{H}^N(B_r(z))}{v_{K,N}(r)}\ \frac{v_{K,N}(r)}{\omega_N r^N}
		\le \vartheta_N[X,\dist,\mathcal{H}^N](z) \frac{3}{2}\le \frac{3}{4}
		\quad\text{for any }0<r<1.
	\end{equation*}
	Using again \cite[Theorem 1.3]{DePhilippisGigli18} we deduce that there exists $\eta(N)>0$ such that
	\begin{equation*}
	z\in \mathcal{S}^{N-1}\setminus\mathcal{S}^{N-2}\implies 	\dist_{GH}(B_r(z),B_r(0^N))\ge r\eta(N)\quad\text{for any }0<r<1,
	\end{equation*}
	therefore $\mathcal{S}^{N-1}\setminus \mathcal{S}^{N-2}\subset \bigcap_{r>0} \mathcal{S}_{\eta(N),r}^{N-1}$. Since the set in the right hand side is closed one has
	\begin{equation*}
	\partial X\subset \bigcap_{r>0} \mathcal{S}_{\eta(N),r}^{N-1}.
	\end{equation*}
	Thus, using \eqref{z60} with $0<\eta\le \eta(N)/2$, we deduce
	\begin{equation}\label{z62}
		\mathcal{H}^N(T_r(\partial X)\cap B_{1/2}(x))\le c(K,N,v,\eta)r^{1-\eta}\quad \text{for any }0<r<\frac{1}{2}\eta\ \text{ and } x\in X.
	\end{equation}
	It is simple to see that, up to increase the constant $c$ one can improve \eqref{z62} obtaining the following statement: for any $\eta>0$ it holds
	\begin{equation}\label{z63}
	\mathcal{H}^N(T_r(\partial X)\cap B_{1/2}(x))\le c(K,N,v,\eta)r^{1-\eta}\quad \text{for any }0<r<\frac{1}{4}\eta(N)\ \text{and } x\in X,
	\end{equation}
	therefore, setting $r(N):=\eta(N)/4$, we have the sought estimate.
	
	\textbf{Step2.} Let us remove the assumption $K>-k$ by means of a covering and scaling argument.\\	
	We can assume without loss of generality that $K<0$.
	Fix $x\in X$ and $s>0$ such that $-Ks^2=k$.
	Arguing as in the proof of \autoref{lemma: tubolar estimate} we can find $x_1,...,x_m$ in $X$ such that $B_{1/2}(x)\subset \cup_{i=1}^mB_{s/2}(x_i)$ and $m$ is bounded by an explicit constant depending only on $N$ and $K$.
	For any $i=1,...,m$ and $\eta>0$ we apply \eqref{z63} to the space $(X,s^{-1}\dist,\mathcal{H}_{s^{-1}\dist}^N)$ obtaining
	\begin{equation}\label{z64}
		\mathcal{H}^N(T_{rs}(\partial X)\cap B_{s/2}(x_i))\le c(k,N,v,\eta) s^N r^{1-\eta}
		\quad \text{for any }0<r<r(N).
	\end{equation}
	Taking the sum over $i=1,\dots,m$ in \eqref{z64} and using the fact that $s$ depends only on $N$ and $K$ we conclude the proof.	
\end{proof}

\subsection{Proof of \autoref{thm:RCDquantitativevolumebound}}
\subsubsection{A lemma in the spirit of quantitative differentiation}

The arrival point of this subsection is \autoref{cor:AllButFiniteScales2} which, roughly speaking, ensures that, on all but a definite number of scales around every point of $X$, the space is as close as we like to the conical structure. To this aim
we need a lemma which, together with the almost rigidity result about metric cones proved in \autoref{FromAlmostVolumeCone}, will give us the sought result. In this lemma we use a technique reminding the general machinery of quantitative differentiation (see \cite{Cheeger12}). We recall here the definition of conicality given in \autoref{def:Conicality}.
\begin{definition}
	Given a metric space $(X,\dist)$, we define the $t$-conicality of the ball $B_r(x)$ as
	\begin{equation}\label{eq:defN}
	\mathcal{N}_t(B_r(x)) \doteq \inf_{\varepsilon > 0} \left\{ \exists \ Z \  \mbox{metric cone and $\RCD(0,N)$ space with tip $z$} \ \mbox{s.t.} \ \dist_{GH}\left(\bar{B}_{\frac{tr}{2}}(x),\bar{B}_{\frac{tr}{2}}(z)\right) \leq \frac{\varepsilon r }{2} \right\}.
	\end{equation}
\end{definition}


\begin{definition}\label{VolumeEnergy}
	Given an $\RCD(K,N)$ m.m.s. $\left(X,\dist,\meas\right)$ for some $K\in \mathbb{R}$ and $N\in[1,+\infty)$, given $x\in X$ and $R>r>0$ we define the $(R,r)$-volume energy around $x$ as
	\begin{equation}
	\mathcal{W}_{R,r}(x)\doteq \log\left(\frac{\meas(B_r(x))}{v_{K,N}(r)}\cdot\left(\frac{\meas(B_R(x))}{v_{K,N}(R)}\right)^{-1}\right).
	\end{equation}
\end{definition}
\begin{remark}
	It follows from \eqref{prop:BishopGromovInequality} that if $(X,\dist,\meas)$ is an $\RCD(K,N)$ space and $x\in X$,
	\begin{equation}
	W_{R,r}(x) \geq 0.
	\end{equation}	
	Moreover, given any $R_1> r_1\geq R_2 > r_2$, it holds
	\begin{equation}\label{PropertyOfVolumeEnergy}
	\mathcal{W}_{R_1,r_2}(x) \geq \mathcal{W}_{R_1,r_1}(x)+\mathcal{W}_{R_2,r_2}(x)
	\end{equation}
	with equality if $r_1=R_2$.
\end{remark}

\begin{lemma}\label{lem:AllButFiniteScales}
	Given $k>1$, $0<\gamma<1/2$, $v>0$ and $\delta>0$, there exists $i_0\doteq i_0(k,\gamma,v,\delta)\in\mathbb{N}$ such that the following holds. 
	If $\left(X,\dist,\mathcal{H}^N\right)$ is a $\ncRCD(K,N)$ space with $N\in[1,+\infty)$ and $K\in\mathbb{R}$ satisfying \eqref{VolumeBound}, then
	 for any $x\in X$
	\begin{equation}\label{Thesis}
	\left| \left\{i \in \mathbb{N} : \mathcal{W}_{k\gamma^i,\gamma^i}(x) > \delta\right\} \right| \leq i_0,
	\end{equation}
	where $\mathcal{W}$ is defined in \eqref{VolumeEnergy}.
\end{lemma} 

\begin{proof}
	Let $x\in X$ and choose $i_1<i_2<\dots<i_n$ natural numbers such that the intervals $[\gamma^{i_1},k\gamma^{i_1}], \dots [\gamma^{i_n},k\gamma^{i_n}]$ are disjoint and $k\gamma^{i_1} < 1$. An iterative application of \eqref{PropertyOfVolumeEnergy} gives
	\begin{equation}\label{SumOfVolumeEnergy}
	\sum_{j=1}^n \mathcal{W}_{k\gamma^{i_j},\gamma^{i_j}}(x) \leq \mathcal{W}_{k\gamma^{i_1},\gamma^{i_n}}(x).
	\end{equation}
	Now, since $k\gamma^{i_1} <1$, by \eqref{prop:BishopGromovInequality} and the volume bound \eqref{VolumeBound} we get
	\begin{equation}
	\frac{\mathcal{H}^N\left(B_{k\gamma^{i_1}}(x)\right)}{v_{K,N}\left(k\gamma^{i_1}\right)} \geq \frac{\mathcal{H}^N(B_1(x))}{v_{K,N}(1)} \geq v, 
	\end{equation}
	and also
	\begin{equation}
	\frac{\mathcal{H}^N\left(B_{\gamma^{i_n}}(x)\right)}{v_{K,N}\left(\gamma^{i_n}\right)} \leq 1
	\end{equation}
	by \autoref{DensityLessThanOne}. Monotonicity of the logarithm tells that
	\begin{equation}
	\mathcal{W}_{k\gamma^{i_1},\gamma^{i_n}}(x) = \log\left(\frac{\mathcal{H}^N\left(B_{\gamma^{i_n}}(x)\right)}{v_{K,N}\left(\gamma^{i_n}\right)}\cdot\left(\frac{\mathcal{H}^N\left(B_{k\gamma^{i_1}}(x)\right)}{v_{K,N}\left(k\gamma^{i_1}\right)}\right)^{-1}\right) \leq \log\frac{1}{v}
	\end{equation}
	so that, by \eqref{SumOfVolumeEnergy}, it follows
	\begin{equation}\label{SumOfVolumeEnergy2}
	\sum_{j=1}^n \mathcal{W}_{k\gamma^{i_j},\gamma^{i_j}}(x) \leq \log\frac{1}{v}.
	\end{equation}	
	Then, denoting by  $\left\lceil x \right\rceil$ the least integer greater than or equal to $x\in \setR$, the conclusion follows from \eqref{SumOfVolumeEnergy2} choosing
	\begin{equation}
	i_0\geq  \left\lceil-\frac{\log k}{\log \gamma}+1\right\rceil\cdot\delta^{-1}\cdot\log\frac{1}{v}+\left\lceil-\frac{\log k}{\log \gamma}+1\right\rceil.
	\end{equation}
	Indeed, if by contradiction we have the opposite inequality in \eqref{Thesis}, then, excluding the first $\left\lceil-\frac{\log k}{\log \gamma}+1\right\rceil$ terms (i.e. working with the $i$'s such that $k\gamma^i<1$) we have
	\begin{equation}\label{Pigeonhole}
	\left|\left\{i \in \mathbb{N} : k\gamma^i<1 \wedge \mathcal{W}_{k\gamma^i,\gamma^i}(x) > \delta\right\}\right|> \left\lceil-\frac{\log k}{\log \gamma}+1\right\rceil\cdot\delta^{-1}\cdot\log\frac{1}{v}
	\end{equation} 
	and then, dividing the set of all the intervals of the form $[\gamma^i,k\gamma^i]$ with $k\gamma^i<1$ in $\left\lceil-\frac{\log k}{\log \gamma}+1\right\rceil$ subsets made of disjoint intervals, a simple pigeonhole with \eqref{Pigeonhole} tells us that there exist $n\geq\delta^{-1}\log\frac{1}{v}$ disjoint intervals $[\gamma^{i_1},k\gamma^{i_1}],\dots,[\gamma^{i_n},k\gamma^{i_n}]$ with $k\gamma^{i_1}<1$ on which $\mathcal{W}_{k\gamma^i,\gamma^i}(x) > \delta$. Combining this observation with \eqref{SumOfVolumeEnergy2} we obtain a contradiction.
\end{proof}

Now we want to prove an analogous of \autoref{FromAlmostVolumeCone} in this setting. We will measure the closeness to a metric cone by means of the notion of conicality introduced in \autoref{def:Conicality}.
\begin{proposition}\label{prop:FromWtoGHclose}
	Let $K\in\mathbb{R}$, $N\geq 2$, $k>1$, $v>0$ and $\varepsilon>0$ be fixed. Then there exists $0<\delta\doteq \delta(K,N,k,v,\varepsilon)<1$ such that the following holds. If $\left(X,\dist,\mathcal{H}^N\right)$ is a $\ncRCD(K,N)$ space satisfying the volume bound \eqref{VolumeBound} and there exist $0<r<\delta$ and $x\in X$ such that
	\begin{equation}
	\mathcal{W}_{kr,r}(x) \leq \delta,
	\end{equation}
	then 
	\begin{equation}\label{ConclusionOnConicality}
	\mathcal{N}_k\left(B_r(x)\right) \leq \varepsilon.
	\end{equation}
\end{proposition}
\begin{proof}
	Note that $\mathcal{W}_{kr,r}(x)\leq \delta$ is equivalent to
	\begin{equation}
	\frac{\meas(B_{kr}(x))}{v_{K,N}(kr)} \geq e^{-\delta}\frac{\meas(B_{r}(x))}{v_{K,N}(r)}.
	\end{equation}
	So that we can choose $\delta\doteq \min\left\{\frac{\delta'}{k},-\log(1-\delta')\right\}$ where $\delta'=\delta'(K,N,k,\varepsilon)$ is given by \autoref{FromAlmostVolumeCone} taking $\eta=k^{-1}$ and $\frac{\varepsilon}{2k}$ in place of $\eps$ in that statement. Then \autoref{FromAlmostVolumeCone} gives \eqref{ConclusionOnConicality}.
\end{proof}


\begin{corollary}[Quantitative conicality]\label{cor:AllButFiniteScales2}
	Given $K\in\mathbb{R}$, $N\geq 2$, $k>1$, $0<\gamma<1/2$, $v>0$ and $\varepsilon>0$, there exists a natural number $j_0\doteq j_0(K,N,k,\gamma,v,\varepsilon)$ such that the following holds. If $\left(X,\dist,\mathcal{H}^N\right)$ is a $\ncRCD(K,N)$ space satisfying the volume bound \eqref{VolumeBound}, then for all $x\in X$
	\begin{equation}\label{Thesis2}
	\left|\left\{i \in \mathbb{N} : \mathcal{N}_k\left(B_{\gamma^i}(x)\right) > \varepsilon\right\}\right|\leq j_0,
	\end{equation}
	where $\mathcal{N}$ is defined in \eqref{def:Conicality}.
\end{corollary}
\begin{proof}
	Let $0<\delta\doteq \delta(K,N,k,v,\varepsilon)<1$ be given by \autoref{prop:FromWtoGHclose} and $i_0\doteq i_0(k,\gamma,v,\delta)$ given by \autoref{lem:AllButFiniteScales}. Then, according to \autoref{prop:FromWtoGHclose} and \autoref{lem:AllButFiniteScales}, 
	\begin{align*}
	\left|\left\{i \in \mathbb{N} : \mathcal{N}_k\left(B_{\gamma^i}(x)\right) > \varepsilon\right\}\right|&\leq \left|\left\{i \in \mathbb{N} : \gamma^i>\delta\right\}\right| + \left|\left\{i \in \mathbb{N} : \mathcal{W}_{k\gamma^i,\gamma^i}(x)>\delta\right\}\right|\\
	&\leq \left\lceil \frac{\log\delta}{\log\gamma}\right\rceil + i_0,
	\end{align*} 
	so that it is sufficient to choose $j_0\geq \left\lceil \frac{\log\delta}{\log\gamma}\right\rceil + i_0$.
	

\end{proof}

\subsubsection{Construction of the covering and conclusion}\label{sub:proofthm}

From now on we fix $x_0\in X$ and our aim is to construct a good covering of $\mathcal{S}^{k}_{\eta,r}\cap B_{\frac{1}{2}}(x_0)$ in order to give a bound on $\mathcal{H}^N\left(\mathcal{S}^{k}_{\eta,r}\cap B_{\frac{1}{2}}(x_0)\right)$. We recall here the definition of conical sets given in \autoref{def:ConicalSet}.


\begin{definition}
	Following \cite{CheegerNaber13a} we define the \emph{ $\varepsilon-(t,r)$ conical set} in $B_{\frac{1}{2}}(x_0)$ as
	\begin{equation}
	C^{\varepsilon}_{t,r} \doteq \{x\in B_{\frac{1}{2}}(x_0) : \mathcal{N}_t(B_r(x))< \eps\},
	\end{equation}
	where $\mathcal{N}$ is defined in \eqref{def:Conicality}. 
\end{definition}


The following lemma, whose proof is postponed to the next subsection, is a key ingredient for the proof of \autoref{thm:RCDquantitativevolumebound}. 

\begin{lemma}[Covering Lemma]\label{lem:CoveringLemma}
	There exists $c_0(N)>1$, such that given any $\eta>0$ and $0<\gamma <1/2$, there exist $\varepsilon_0\doteq\varepsilon_0(N,K,\gamma,\eta)>0$ and $n_0(N,K,\gamma,\eta)\in \mathbb{N}$ such that the following holds. If for some natural $n_0< j \in\mathbb{N}$ and $k\leq N-1$ we have $x\in \mathcal{S}^k_{\eta,2^{-1}\gamma^{j-1}}\cap B_{\frac{1}{2}}(x_0)$ and  $\mathcal{N}_{\gamma^{-N}}(B_{\gamma^{j-1}}(x)) \leq \varepsilon_0$
	then the minimal number of balls of radius $2^{-1}\gamma^j$ to cover $B_{2^{-1}\gamma^{j-1}}(x)\cap \mathcal{S}^k_{\eta,2^{-1}\gamma^j}\cap C^{\varepsilon_0}_{\gamma^{-N},\gamma^{j-1}}$ is less than $c_0\gamma^{-k}$.
\end{lemma}

\begin{proof}[Proof of \autoref{thm:RCDquantitativevolumebound}]
We can reduce ourselves to prove the sought estimate with $r=2^{-1}\gamma^j$ for every $j\in\mathbb{N}$, for a fixed $0<\gamma(K,N,\eta)<1/2$ which will be chosen later.
Indeed, suppose that there exist $0<\gamma(K,N,\eta)<1/2$ and $c(K,N,v,\eta)$ such that, for every $j\in\mathbb{N}$,
\begin{equation}\label{ReductionToGamma}
\mathcal{H}^N\left(\mathcal{S}^k_{\eta,2^{-1}\gamma^j}\cap B_{\frac{1}{2}}(x_0)\right) \leq c(2^{-1}\gamma^j)^{(N-k-\eta)}.
\end{equation}
Then, given $0<r<1/2$, we can find $j$ such that $2^{-1}\gamma^{j+1} < r \leq 2^{-1}\gamma^j$. Since $s\to\mathcal{S}^k_{\eta,s}$ is increasing, we easily obtain 
\begin{equation*}
\mathcal{H}^N\left(\mathcal{S}^k_{\eta,r}\cap B_{\frac{1}{2}}(x_0)\right) \leq \bar{c}r^{N-k-\eta}
\end{equation*}
with $\bar c(K,N,v,\eta)\doteq c(K,N,v,\eta)\gamma(K,N,\eta)^{-(N-k-\eta)}$.

Let us prove \eqref{ReductionToGamma}.
From now on we will denote any $j$-uple with entries in $\{0,1\}$ with $T^j$ and the $i$-th entry of this $j$-uple with $T^j_i$. Also $|T^j|$ will indicate the number of $1$'s in this $j$-uple.
Let us fix $j\in\mathbb{N}$. To each $x\in B_{\frac{1}{2}}(x_0)$ we can associate $T^j(x)$ a $j$-uple with entries in $\{0,1\}$ as follows: for $i\leq j$
\begin{equation}\label{DefinitionJuple}
T_i^j(x) = 0 \Leftrightarrow x \in C^{\varepsilon}_{\gamma^{-N},\gamma^i}.
\end{equation}
For any $j$-uple $T^j$ with entries in $\{0,1\}$ we let
\begin{equation}
E_{T^j}\doteq \{x\in B_{\frac{1}{2}}(x_0) : T^j(x)=T^j\}.
\end{equation}
An immediate consequence of \autoref{cor:AllButFiniteScales2} is that if $E_{T^j}$ is not empty for some $j$-uple $T^j$, then 
\begin{equation}\label{BoundOnTj}
|T^j|\leq j_0\doteq j_0(K,N,\gamma,v,\varepsilon).
\end{equation}
Indeed, if $E_{T^j}$ is not empty, then there exists $x\in B_{\frac{1}{2}}(x_0)$ such that $T^j(x)=T^j$. Recalling that a $j$-uple defined starting from a point according to \eqref{DefinitionJuple} has a 1 in the $i$-th entry if and only if $\mathcal{N}_{\gamma^{-N}}\left(B_{\gamma^i}(x)\right)\geq \varepsilon$, the estimates of \autoref{cor:AllButFiniteScales2} applied with $k=\gamma^{-N}$, gives the sought result.\\
The bound obtained in \eqref{BoundOnTj} allows to estimate the number of non empty sets $E_{T^j}$ by $2j^{j_0}$. Indeed, the number of possible choices of $j_0$ positions in a string with $j\geq j_0$ entries is 
\begin{equation*}
\binom{j}{j_0}\le 2j^{j_0},
\end{equation*}
and the estimate holds also in the case $j<j_0$ since in that case the $j$-uples are at most $2^j$ which is less than the right hand side in the previous equation since $j<j_0$.

Let us define now inductively on $j$ the covering of $\mathcal{S}^k_{\eta,2^{-1}\gamma^j}\cap B_{\frac{1}{2}}(x_0)$ in such a way that
\begin{equation}
\mathcal{S}^k_{\eta,2^{-1}\gamma^j}\cap B_{\frac{1}{2}}(x_0) \subseteq\bigcup_{T^j:E_{T^j}\neq\emptyset} \mathcal{B}^{T^j},
\end{equation}
where $\mathcal{B}^{T^j}$ is a union of balls of radius $2^{-1}\gamma^j$.\\ 
For $j=1$ we let $\mathcal{B}^{(0)}$ be the union of the minimum amount of balls of radius $2^{-1}\gamma$ with centers in $\mathcal{S}^k_{\eta,2^{-1}\gamma}\cap E_{(0)}$ needed to cover $\mathcal{S}^k_{\eta,2^{-1}\gamma}\cap E_{(0)}$, if this intersection is not empty. Then we let $\mathcal{B}^{(1)}$ be the union of the minimum amount of balls of radius $2^{-1}\gamma$ with centers in $\mathcal{S}^k_{\eta,2^{-1}\gamma}\cap E_{(1)}$ which we need to cover $\mathcal{S}^k_{\eta,2^{-1}\gamma}\cap E_{(1)}$, if this intersection is not empty.\\
Now for any $j>1$ and for any $T^j$ for which $E_{T^j}$ is not empty, we want to define $\mathcal{B}^{T^j}$. Let us consider the $(j-1)$-uple $T^{j-1}$ which we obtain by dropping the last entry in $T^j$. For each ball $B_{2^{-1}\gamma^{j-1}}(\bar x)$ in $\mathcal{B}^{T^{j-1}}$, we take the minimum amount of balls of radius $2^{-1}\gamma^j$ with centers in $\mathcal{S}^k_{\eta,2^{-1}\gamma^{j}}\cap E_{T^j}\cap B_{2^{-1}\gamma^{j-1}}(\bar x)$ needed to cover $\mathcal{S}^k_{\eta,2^{-1}\gamma^{j}}\cap E_{T^j}\cap B_{2^{-1}\gamma^{j-1}}(\bar x)$, if this intersection is not empty.

The next step in order to achieve the volume estimate \eqref{eq:volumebound} aims to bound the cardinality of the families $\mathcal{B}^{T^j}$. We claim that for any such family, setting $Q\doteq n_0+j_0$, the number of balls needed can be controlled by
\begin{equation}\label{eq:boundingballs}
\left( c_1\gamma^{-N}\right)^Q\cdot\left(c_0\gamma^{-k}\right)^{j-Q},
\end{equation}
for some constants $c_1(N,K)\ge c_0(N)>1$. To this aim we just observe that \eqref{eq:boundingballs} follows from the way in which we constructed the covering, after the appropriate choice of $\varepsilon_0$ forced by \autoref{lem:CoveringLemma}, by means of an induction argument. Indeed the factor with exponent $Q$ in \eqref{eq:boundingballs} arises from the at most $j_0+n_0$ scales on which the assumptions of \autoref{lem:CoveringLemma} are not satisfied and therefore we are forced to cover with $c_1\gamma^{-N}$ balls (this possibility is granted by \eqref{prop:BishopGromovInequality}). The factor with exponent $j-Q$ instead arises from the remaining scales on which \autoref{lem:CoveringLemma} applies and we can cover with less than $c_0\gamma^{-k}$ balls.

Recapitulating what we obtained so far, we proved that there exist constants $c_1(K,N)\ge c_0(N)>1$ and a natural number $j_0$ such that, for any natural $j$, the set $\mathcal{S}^{k}_{\eta,2^{-1}\gamma^j}\cap B_{\frac{1}{2}}(x_0)$ is contained in the union of at most $2j^{j_0}$ non empty families of balls. Furthermore, each of the families above contains at most $(c_1\gamma^{-N})^Q(c_0\gamma^{-k})^{j-Q}$ balls of radius $2^{-1}\gamma^j$.\\
Let us see how \eqref{eq:volumebound} can be obtained starting from these results. First we let $\gamma=\gamma(\eta)\doteq c_0^{-\frac{2}{\eta}}$, where $c_0$ is given by \autoref{lem:CoveringLemma}. Then we observe that $c_0^{j}=\left(\gamma^j\right)^{-\frac{\eta}{2}}$, $j^{j_0}\le c(N,K,v,\eta)(\gamma^j)^{-\frac{\eta}{2}}$ and up to choose $\eta$ small enough $0<\gamma <1/2$. The considerations above, together with the volume comparison yielding $\mathcal{H}^N(B_{2^{-1}\gamma^j}(x))\le c_2(N,K)(2^{-1}\gamma^j)^{N}$, give the estimate
\begin{align*}
\mathcal{H}^N\left(\mathcal{S}^{k}_{\eta,2^{-1}\gamma^j}\cap B_{\frac{1}{2}}(x_0)\right)\le& 2j^{j_0}\left[(c_1\gamma^{-N})^Q\cdot(c_0\gamma^{-k})^{j-Q} \right]\cdot c_2\cdot(2^{-1}\gamma^j)^N\\
\le& c_3(N,K,v,\eta)\cdot j^{j_0}\cdot c_0^{j}\cdot(\gamma^j)^{N-k}\\
\le& c_4(N,K,v,\eta)\cdot(\gamma^{j})^{N-k-\eta}.
\end{align*}
In view of what we observed at the beginning of the proof, the estimate above gives the desired result when $\eta$ is small enough, this in turn implies the general case thanks to \eqref{eq:inclusion singular set}.
\end{proof}

\subsubsection{Proof of the covering lemma via cone splitting}

Aim of this subsection is to prove \autoref{lem:CoveringLemma}. The key tool in proving it will be the effective almost cone splitting theorem proved in \autoref{subsec:almostcone splitting} that we restate here for the reader convenience.

\begin{theorem}[Cone splitting, quantitative version]
For all $K\in \mathbb{R}$, $N\in [2,+\infty)$, $0<\gamma<1$, $\delta<\gamma^{-1}$, and for all $\tau,\psi>0$ there exist $0<\varepsilon(N,K,\gamma,\delta,\tau,\psi)<\psi$ and $0<\theta=\theta(N,K,\gamma,\delta,\tau,\psi)$ such that the following holds. Let $(X,\dist,\meas)$ be an $\RCD(K,N)$ m.m.s., $x\in X$ and $r\le\theta$ be such that there exists an $\varepsilon r$-GH equivalence
\begin{equation*}
F:B_{\gamma^{-1}r}\left((0,z^*)\right)\to B_{\gamma^{-1}r}(x)
\end{equation*}
for some cone $\setR^l\times C(Z)$, with $(Z,\dist_Z,\meas_Z)$ $\RCD(N-l-2,N-l-1)$ m.m.s..
If there exists 
\begin{equation*}
x'\in B_{\delta r}(x)\cap C^{\varepsilon}_{\gamma^{-N},\delta r}
\end{equation*}
with 
\begin{equation*}
x'\notin T_{\tau r}\left(F(\setR^l\times\left\lbrace z^*\right\rbrace )\right)\cap B_r(x),
\end{equation*}
then for some cone $\setR^{l+1}\times C(\tilde{Z})$, where  $(\tilde{Z},\dist_{\tilde{Z}},\meas_{\tilde{Z}})$ is a $\RCD(N-l-2,N-l,1)$ m.m.s.,
\begin{equation*}
\dist_{GH}\left(B_r(x), B_r((0,\tilde{z}^*))\right)<\psi r.
\end{equation*}  
\end{theorem}

\begin{corollary}\label{cor:almostconicalnear}
For all $K\in \mathbb R$, $N\in[2,+\infty)$, $k\leq N-1$, $0<\gamma<1/2$, $\eta>0$ and for all $\tau,\psi>0$ there exist $\eps(K,N,\gamma,\eta,\tau,\psi)>0$ and $\theta(K,N,\gamma,\eta,\tau,\psi)>0$ such that, for any $\RCD(K,N)$ m.m.s. $(X,\dist,\meas)$, the following holds. Let $r\le\theta$ and $x\in C^{\eps}_{\gamma^{-N},r}\cap \mathcal{S}^k_{\eta, 2^{-1} r}$. Then there exist a cone $\setR^l\times C(\tilde{Z})$ with $l\leq k$, an $\RCD(N-l-2,N-l-1)$ m.m.s. $\tilde{Z}$ and a $(\frac{\psi}{2} r)$-GH equivalence
\begin{equation*}
F:B_{r/2}((0,\tilde{z}^*))\to B_{r/2}(x)
\end{equation*}
such that
\begin{equation}\label{eq:usefulinclusion}
C^{\eps}_{\gamma^{-N}, r}\cap B_{r/2}(x)\subset T_{2^{-1}\tau r}\left(F(\setR^l\times\{\tilde{z}^*\})\right)
\end{equation}
\end{corollary}

\begin{proof}
Let $\varepsilon^{[N]}\doteq\varepsilon(N,K,\gamma,1,\tau,\psi)$ and $\theta^{[N]}\doteq\theta(N,K,\gamma,1,\tau,\psi)$ be given by \autoref{thm:conesplittingquant} and inductively define, still by \autoref{thm:conesplittingquant}, for all $0\leq l \leq N-1$, $$
\varepsilon^{[l]}\doteq\varepsilon(N,K,\gamma,2\gamma^{N-l-1},\tau\gamma^{N-l-1},\gamma^{N-l-1}\varepsilon^{[l+1]})\footnote{Here we use $1/2>\gamma$ in order to obtain $2\gamma^{N-l-1}<\gamma^{-1}$ for all $0\leq l\leq N-1$ and apply \autoref{thm:conesplittingquant}.},$$
$$ 
 \theta^{[l]}\doteq\theta(N,K,\gamma,2 \gamma^{N-l-1},\tau\gamma^{N-l-1},\gamma^{N-l-1}\varepsilon^{[l+1]}).
$$
Observe that $\varepsilon^{[0]}<\varepsilon^{[1]}<\cdots<\varepsilon^{[N]}<\psi$ and put $\eps\doteq\varepsilon^{[0]}$. Choose $\theta=\min_{l}\left\{\frac{2\theta^{[l]}}{\gamma^{-(N-l-1)}}\right\}$. \\
By assumption $x\in C^{\eps}_{\gamma^{-N},r}$, hence we can find the largest $0\le l\le N$ such that for some cone $\setR^{l}\times C(\tilde{Z})$, with $\tilde{Z}$ an $\RCD(N-l-2,N-l-1)$ m.m.s., there is an $(2^{-1}\varepsilon^{[l]}r)$-GH equivalence $F:B_{2^{-1}\gamma^{-(N-l)}r}((0,\tilde{z}^*))\to B_{2^{-1}\gamma^{-(N-l)}r}(x)$. Note that we can assume $l\leq k\le N-1$. Indeed,
 since it is not restrictive to take $\varepsilon^{[l]}<\eta$, we then have a $(2^{-1}\eta r)$-GH equivalence between $B_{r/2}((0,\tilde{z}^*))$ and $B_{r/2}(x)$, which is impossible if $l>k$ since $x\in\mathcal{S}^k_{\eta, 2^{-1}r}$. Applying \autoref{thm:conesplittingquant} with $r'=2^{-1}\gamma^{-(N-l-1)}r$, $\delta'=2\gamma^{N-l-1}$, $\tau'=\tau \gamma^{N-l-1}$ and $\psi'=\gamma^{N-l-1}\eps^{[l+1]}$ and considering $l\leq k\leq N-1$\footnote{This is important to see that an $2^{-1}\varepsilon^{[l]} r$-GH equivalence is a $2^{-1}\varepsilon^{[l]}\gamma^{-(N-l-1)}r$-GH equivalence when we apply \autoref{thm:conesplittingquant}.}, we obtain
$$
B_{ r}(x)\cap C^{\varepsilon^{[l]}}_{\gamma^{-N}, r} \subseteq T_{2^{-1}\tau r}\left(F(\mathbb R^l\times\{\tilde{z}^*\})\right).
$$
Now the conclusion comes from the straightforward inclusion 
$$
B_{r/2}(x)\cap C^{\eps}_{\gamma^{-N}, r} \subseteq B_{r}(x)\cap C^{\varepsilon^{[l]}}_{\gamma^{-N},r}.
$$
\end{proof}

Finally we can pass to the proof of \autoref{lem:CoveringLemma}. 

\begin{proof}[Proof of \autoref{lem:CoveringLemma}]

Let us choose  $\varepsilon_0=\delta\left(K,N,\gamma,\eta,\frac{1}{10}\gamma,\frac{1}{10}\gamma\right)$ and \\ $\theta=\theta\left(K,N,\gamma,\eta,\frac{1}{10}\gamma,\frac{1}{10}\gamma\right)$ as in the previous corollary. Let $n_0(K,N,\gamma,\eta)$ be a sufficiently big natural number so that $\gamma^{j-1}\leq \theta$ for all $j\geq n_0$. 
Then as we are in the hypothesis $x\in C^{\varepsilon_0}_{\gamma^{-N},\gamma^{j-1}}\cap\mathcal{S}^k_{\eta,2^{-1}\gamma^{j-1}}$, we can apply the previous corollary with $r=\gamma^{j-1}$ to obtain $F$ a $\frac{1}{20}\gamma^j$-GH equivalence between the ball $B_{2^{-1}\gamma^{j-1}}(x)$ and some ball of the same radius in a metric cone $\mathbb R^{l}\times C(\tilde{Z})$ with $l\leq k$. We also obtain

\begin{equation}\label{eq:improvedinclusion}
B_{2^{-1}\gamma^{j-1}}(x)\cap\mathcal{S}^{k}_{\eta,2^{-1}\gamma^j}\cap C^{\varepsilon_0}_{\gamma^{-N},\gamma^{j-1}}\subset T_{\frac{1}{20}\gamma^j}(F(\setR^l\times\{\tilde{z}^*\}))\cap B_{2^{-1}\gamma^{j-1}}(x),
\end{equation} 
and then the sought estimate about the number of balls of radius $2^{-1}\gamma^j$ necessary to cover $B_{2^{-1}\gamma^{j-1}}(x)\cap\mathcal{S}^{k}_{\eta,2^{-1}\gamma^j}\cap C^{\varepsilon_0}_{\gamma^{-N},\gamma^{j-1}}$ follows from \eqref{eq:improvedinclusion} and the observation that in the Euclidean space $\setR^k$ the number of balls of radius $\gamma^{j}$ needed to cover a ball of radius $\gamma^{j-1}$ can be controlled by $c\gamma^{-k}$, for some dimensional constant $c=c(k)>0$.

\end{proof}

\end{document}